\documentclass[10pt,twocolumn]{IEEEtran}
\usepackage{bigints}
\usepackage{abraces}
\usepackage{bbold}
\usepackage[noadjust]{cite}

\usepackage{amsmath,amssymb,amsfonts}
\usepackage{algorithmic}
\usepackage{graphicx}
\usepackage{textcomp}
\usepackage{color}
\definecolor{darkgreen}{rgb}{0,0.4,0}
\usepackage{float}
\definecolor{darkgreen}{rgb}{0,0.6,0}
\usepackage{authblk}
\usepackage[colorlinks,linkcolor=blue,citecolor=darkgreen]{hyperref}
\usepackage[latin1]{inputenc}
\usepackage{amsmath}
\usepackage{graphicx}
\usepackage{amssymb}
\usepackage{verbatim}
\usepackage{epsfig}
\usepackage[labelfont=bf]{caption}
\usepackage{enumerate}
\usepackage{subfig}
\usepackage{epstopdf}
\usepackage{relsize,exscale}

\usepackage{booktabs} 
\newtheorem{theorem}{Theorem}
\newtheorem{definition}{Definition}
\newtheorem{assumption}{Assumption}
\newtheorem{lemma}{Lemma}
\newtheorem{remark}{Remark}

\newtheorem{proof}{Proof}
\newtheorem{proposition}{Proposition}

\def\begquo{\begin{quote}}
\def\endquo{\end{quote}}
\def\begequarr{\begin{eqnarray}}
\def\endequarr{\end{eqnarray}}
\def\begequarrs{\begin{eqnarray*}}
\def\endequarrs{\end{eqnarray*}}
\def\begarr{\begin{array}}
\def\endarr{\end{array}}
\def\begequ{\begin{equation}}
\def\endequ{\end{equation}}

\def\begdes{\begin{description}}
\def\enddes{\end{description}}
\def\begenu{\begin{enumerate}}
\def\begite{\begin{itemize}}
\def\endite{\end{itemize}}
\def\endenu{\end{enumerate}}

\def\lef[{\left[\begin{array}}
\def\rig]{\end{array}\right]}
\def\begcen{\begin{center}}
\def\endcen{\end{center}}
\def\begrem{\begin{remark}\rm}
\def\endrem{\end{remark}}
\def\begdef{\begin{definition}}
\def\enddef{\end{definition}}
\def\begpro{\begin{proposition}}
\def\endpro{\end{proposition}}
\def\begfac{\begin{fact}}
\def\endfac{\end{fact}}
\def\begass{\begin{assumption}}
\def\endass{\end{assumption}}
\def\begsubequ{\begin{subequations}}
\def\endsubequ{\end{subequations}}

\def\begmat#1{\begin{bmatrix}#1\end{bmatrix}}
\def\begpma#1{\begin{pmatrix}#1\end{pmatrix}}

\def\cale{{\cal E}}
\def\cali{{\cal I}}

\def\calh{{\cal H}}
\def\calc{{\cal C}}

\def\cale{{\cal E}}

\def\call{{\cal L}}

\def\cala{{\cal A}}

\def\bfx{{\bf x}}

\def\dlc{\delta \chi}

\def\phil{\phi^{\mathtt{L}}}
\def\phile{\phi^{\mathtt{L}}_e}
\def\rhol{\rho^{\mathtt{L}}}

\def\L2e{{\cal L}_{2e}}

\def\rea{\mathbb{R}}

\def\diag{\mbox{diag}}

\def\col{\mbox{col}}
\def\hal{{1 \over 2}}

\def\diag{\mbox{diag}}

\def\im{\mbox{Im}\;}

\def\min{{\mbox{min}}}
\def\max{{\mbox{max}}}

\def\IJRNLC{{\it Int. J. Robust and Nonlinear Control}}
\def\TAC{{\it IEEE Trans. Automatic Control}}

\def\CDC{{\it IEEE Conf. Decision and Control}}
\def\SCL{{\it Systems \& Control Letters}}
\def\AUT{{\it Automatica}}

\def\SIAM{{\it SIAM J. Control and Optimization}}

\def\ACC{{\it American Control Conf.}}

\usepackage{color}
\newcommand{\blue}[1]{{\color{blue} #1}}

\def\QED{\hfill $\square$}
\def\qed{\hfill $\triangleleft$}

\usepackage[prependcaption,colorinlistoftodos]{todonotes}

\input epsf

\def\BibTeX{{\rm B\kern-.05em{\sc i\kern-.025em b}\kern-.08em
    T\kern-.1667em\lower.7ex\hbox{E}\kern-.125emX}}
\begin{document}

\title{Reduced-Order Nonlinear Observers via Contraction Analysis and Convex Optimization}

\author{Bowen Yi, Ruigang Wang, and Ian R. Manchester
%
\thanks{This work was supported by Australian Research Council.}
\thanks{The authors are with Australian Centre for Field Robotics and Sydney Institute for Robotics and Intelligent Systems, The University of Sydney, Sydney, NSW 2006, Australia. (email: \texttt{\{bowen.yi,ruigang.wang,ian.manchester\}@sydney.edu.au})
}
}

\maketitle

\begin{abstract}
In this paper, we propose a new approach to design globally convergent reduced-order observers for nonlinear control systems via contraction analysis and convex optimization. Despite the fact that contraction is a concept naturally suitable for state estimation, the existing solutions are either local or relatively conservative when applying to physical systems. To address this, we show that this problem can be translated into an off-line search for a coordinate transformation after which the dynamics is (transversely) contracting. The obtained sufficient condition consists of some easily verifiable differential inequalities, which, on one hand, identify a very general class of ``detectable'' nonlinear systems, and on the other hand, can be expressed as computationally efficient convex optimization, making the design procedure more systematic. Connections with some well-established approaches and concepts are also clarified in the paper. Finally, we illustrate the proposed method with several numerical and physical examples, including polynomial, mechanical, electromechanical and biochemical systems.
\end{abstract}


%
\section{Introduction}
\label{sec1}
%
Online state estimation of dynamical systems is of both practical and theoretical importance, since some system states are important for the purpose of control or monitoring but often unavailable due to technological or cost constraints. Apart from the control community, similar problems also arise in many other fields, {\em e.g.}, time-series prediction, machine learning and signal processing.

For linear systems, there are comprehensive results developed for observer design, see \cite{LUEtac71} for the asymptotic case, and recent papers \cite{ENGKRE,ORT2020} achieving finite-time convergence. In contrast, constructive nonlinear observer is a more challenging task, which has been intensively studied for several decades in the control community, producing more than six research monographs on this topic \cite{BES,GAUKUPbook,ASTbook,NIJFOSbook,BER,KHAbook}. Despite many efforts, there are few tools available to design globally or semi-globally convergent observers for nonlinear \emph{control} (a.k.a. time-varying or non-autonomous) systems, with notable exception of high-gain observers \cite{KHAbook} and recently results on Kazantzis-Kravaris-Luenberger (KKL) observers \cite{BERAND}. To the best of our knowledge, these two methods are the most general and systematic frameworks for nonlinear observer design at present. In high-gain observers, a triangular form of the system dynamics with strong differential observability is required to achieve semi-global convergence by high-gain injection, which, however, is widely recognized as harmful in engineering practice with the deleterious effect of amplification of unavoidable high-frequency noise \cite{KHAPRA}. KKL observers are motivated by Luenberger's initial idea by adopting a coordinate change in order to get a linear stable error dynamics. Such an idea was extended to the nonlinear context in \cite{SHO,KAZKRA}, later elaborated in \cite{KREENG,ANDPRA}, and finally applied to non-autonomous systems in \cite{BERAND}, which provides an elegant theoretical framework to state estimation. From the numerical perspective, the last step in KKL observers involves \emph{online} computation of an inverse mapping of smoothing operators, which can be formalized as non-convex online optimization, making it a duanting task to be solved in real time, particularly for high-dimensional systems. Besides, there are some recently developed observer design tools applicable to different particular classes of nonlinear systems, {\em e.g.}, immersion and invariance (I\&I) observer \cite{ASTbook}, and parameter estimation-based observer \cite{ORTscl,YIetal}. 

In this paper, we aim to provide a novel constructive approach to nonlinear observers via contraction analysis, a concept which was introduced to the control community in \cite{LOHSLO}, and may historically date back to \cite{LEW}. Contraction, also known as incremental exponential stability, studies the convergence between any pairs of solutions rather than a particular one, by studying the differential dynamics. Its extensions to constructive nonlinear feedback are studied in \cite{MANSLOtac}, see also \cite{TOBetal} for its application in system identification. As clearly illustrated in \cite{LOHSLO}, observer design is one of the initial motivations for contraction analysis. Indeed, using contraction analysis in observer design is not a new idea, and has similarities to the seminal work \cite{LUE}, where a coordinate change is involved in order to get a stable error dynamics. In \cite{TARRAS}, a technique similar to contraction analysis was used for constructive nonlinear stochastic observers. More recently, some necessary conditions on the existence of asymptotically convergent observers in the \emph{original coordinate} for general nonlinear autonomous systems are proposed in \cite{SANPRA}, which is motivated by the study of contracting flows in Riemannian manifolds. It shows that the existence of a convergent observer for autonomous systems implies that the system vector field is strictly geodesically monotonic, tangentially to the output function level sets. However, from a constructive viewpoint the result in \cite{SANPRA} can only guarantee local convergence, unless imposing a totally geodesic assumption, which is generally difficult to verify for a given system. There are also some results using contraction to design observers for specific classes of systems, {\em e.g.}, \cite{AGHROU} for mechanical systems and \cite{YAGYAZ} for port-Hamiltonian systems, whose nature is intrinsically local. 

With these considerations, the following questions arise.
\begin{itemize}
    \item[1)] What can be the sufficient conditions to design a \emph{globally} convergent observer for nonlinear control systems by studying contracting flows? Though an answer has already been partially reported in our previous work \cite{MAN}, we here consider reduced-order observer design and show that coordinate transformation  provides sufficient degrees of freedom, in this way sharpening the existing results. 
    \item[2)] Is it possible to translate the obtained condition into off-line convex optimization? When designing globally convergent observers, it is quite common to encounter partial differential equations (PDEs) or inequalities, thus requiring the user familiar with system structures or invoking some physical insights. To circumvent the difficulty, we give a convex representation in the paper, making it possible to take advantage of powerful numerical tools.
    \item[3)] Is the set of systems identified in the paper large enough? To answer this question, we give links between the proposed method and some well-established concepts, or methods, including in the linear system context, strong differential observability, I\&I observers and transverse contraction. As illustrations, we will show how to use the proposed framework to design reduced-order observers for several benchmark physical systems---the magnetic levitation (MagLev) model, the cart-pendulum system, and a biological reactor---and an academic example to verify the results with convex constraints. 
\end{itemize}

The remainder of the paper is organized as follows. In Section \ref{sec2} we formulate the problem of reduced-order observer design and give some preliminaries on contraction analysis. In Section \ref{sec:3} we start with the basic case that the observer shares the dimension of the unknown states, and present our general design framework. It is followed by the convex representation, as well as the generalized result removing the dimension constraint. In Section \ref{sec4}, we present further discussions on the proposed method and compare it with some existing methods. Some examples with simulation results are given in Section \ref{sec5}. The paper is wrapped up with some concluding remarks.

{\em Caveat.} An abridged conference version of the paper can be found in \cite{YIacc}.

%
\section{Problem Formulation and Preliminaries}
\label{sec2}
%

\subsection{Notations and symbols}
{\em Notations.} All mappings are smooth enough, and when clear from the context, the arguments of mappings and the subindex of operators will be omitted. Given a matrix $P(x)$ and a vector field $f(x)$ with proper dimensions, we define the directional derivative as $\partial_f P(x) = \sum_j {\partial P(x) \over \partial x_j} f_j(x)$. For a full-rank matrix $g\in \rea^{n\times m}$ ($m<n$), we denote the generalized inverse as $g^\dagger  = [g^\top g]^{-1} g^\top$ and $g^\bot$ a full-rank left-annihilator such that $g^\bot g =0$. We use $|\cdot|$ to denote Euclidean norms of vectors or matrices, and $B_\epsilon(\cala)=\{x\in \rea^n|\inf_{y\in\cala}|x-y|\le \epsilon\}$ for a compact set $\cala \subset \rea^n$ and some $\epsilon>0$. We use $\mbox{cl}(\cdot)$ to represent the closure of a set, and $\mbox{Im}[\cdot]$ denotes the image of a given function. Given $f: \rea^n \to \rea$ we define the differential operator $\nabla f := ({\partial f \over \partial x})^\top$, and use $f(x)|_{x_b}^{x_a} = f(x_a) - f(x_b)$. For square matrices $A$ and $B$, the notation $A \succeq B$ indicates $(A-B)$ positive semidefinite, and we use $\texttt{sym}\{A\}$ to represent the symmetric part $\hal(A+A^\top)$ of $A$. The set of all polynomials in $x\in \rea^n$ with real coefficients of degree $m \in \mathbb{N}_+$ is written as $\rea_m[x]$. A matrix-valued function $A:\rea^n \to \rea^{m\times m}_{\succ 0}$ is called uniformly bounded if $a_1 I \preceq A(x) \preceq a_2 I, \; \forall x$ with some $a_2\ge a_1>0$. $\lambda_i\{\cdot\}$s denote the eigenvalues of square matrices. We use $\phil$ to denote the left inverse of a mapping $\phi$, {\em i.e.}, $\phil\circ \phi$ is an identity mapping.

{\em Nomenclature.} The variables $x,y,u$ represent the unknown state, measurable state (output) and the input, respectively, which all live in Euclidean space, and we also write $\chi:=\col(x,y)$. For a variable $(\cdot)$, we assume its dimension being $n_{(\cdot)}$. For the dynamics of $x$, we use $X(t;x_0)$ to denote its solution from the initial condition $x_0$ at $t=0$.


%
\subsection{Problem formulation}
\label{sec21}

In this paper we consider state estimation of nonlinear control (time-varying) systems in the form of
\begin{equation}
\label{eq:syst}
\begin{aligned}
	\dot x  &= f_x(x,y,u), \\
	\dot y  &= f_y(x,y,u),
\end{aligned}
\end{equation}
with unknown state $x \in \rea^{n_x}$, measured state $y \in \rea^{n_y}$ and input $u \in \rea^{n_u}$. We sometimes compactly write as
\begin{equation}
\label{chi}
\dot \chi = f(\chi,u)	
\end{equation}
and $n_\chi: = n_x+n_y$. Following the standard practice in observer design, we assume the considered input signal $u:\rea_+ \to \rea^{n_u}$ belonging to a set ${\cal U} \subset \call_\infty^{n_u}$, such that the system \eqref{eq:syst} is forward complete. Namely, the solution exists uniquely for $t\in [0,\infty)$.

{\em Problem 1.} The objective is to design an observer
\begin{equation}
\label{observer:general}
\begin{aligned}
	\dot{\xi} & = N(\xi, y, u),
	\\
	\hat x   & = H (\xi, y)
\end{aligned}
\end{equation}
with $\xi \in \rea^{n_\xi}$ the observer state, and $\hat x \in \rea^{n_x}$ the observer output, guaranteeing
\begin{equation}
\label{conv:general}
 \lim_{t\to\infty} |\hat x(t) - x(t)| =0.
\end{equation}
If there exists a function $b: \rea^{n_x} \times \rea^{n_x} \to \rea_{\ge 0}$ such that
$$
|\hat x(t) - x(t) |\le e^{- \lambda t} b(\hat{x}(0),x(0)), \quad b(x,x) =0,~ \lambda>0,
$$
then we call \eqref{observer:general} an exponential observer with rate $\lambda$. When only estimating the unknown part $x$ rather than the full system state $(x,y)$, we refer it as a reduced-order observer. In particular, the case $n_x \le n_\xi < n_\chi$ is clearly reduced-order.

\begin{remark}\rm
\label{rem1}
We consider the class of systems \eqref{eq:syst} with the output partial states of the system. On one hand, this coordinate simplifies the following analysis in reduced-order observers design; on the other hand, it is quite general to cover several observable canonical forms and plenty of physical models, since the model $\dot \bfx =f (\bfx,u)$, $y =h(\bfx)$ can be transformed into \eqref{eq:syst} by finding the complementary coordinates in many cases assuming that $\nabla h(\bfx)$ is full rank. 
\end{remark}

\subsection{Preliminaries on contraction analysis}

Contraction analysis allows us to obtain convergence independent of knowledge of a particular solution \cite{LOHSLO,FORSEP}. In contrast, a pre-defined solution is required in standard Lyapunov methods, and one then needs to construct a Lyapunov function of observation errors \cite{TSI,PRA}. This fact makes contraction analysis particularly tailored for observer design. Let us recall some definitions first.
 
\begin{definition}
\label{def:contraction}\rm ({\em Contraction})
Assume the nonlinear system 
\begin{equation}
\label{NLS}
\dot \bfx = f(\bfx, t), \quad \bfx\in \rea^{n}
\end{equation}
is forward complete with the solution $X(t;\bfx_0)$ from $\bfx(0)=\bfx_0$. The system is
\begin{itemize}
 \item {}asymptotically contracting if $\forall (x_a,x_b) \in \rea^{n}\times \rea^{n}$,
 $$
 |X(t;x_a)-X(t;x_b)| \le \kappa(|x_a-x_b|,t)
 $$
 holds for any $t\ge 0$ and some function $\kappa$ of class $\mathcal{KL}$.
 \item {}contracting if the system is asymptotically contracting with $\kappa(a,t) = k e^{-\lambda t}a $ for some $k,\lambda>0$.
\end{itemize}
\end{definition}

In contraction analysis we need to calculate the differential behaviour of an infinitesimal displacement $\delta\bfx$ along the flow $X(t;\bfx(0))$ generated by \eqref{NLS}, then obtaining the differential dynamics
\begequ
\label{dot_deltaxi}
{d\over dt}\delta\bfx = {\partial f \over \partial \bfx}(\bfx, t)\delta \bfx,
\endequ
which can be regarded as the first-order approximation along solutions. Then, the incremental stability of \eqref{NLS} is characterized by the stability of the linear time-varying (LTV) differential dynamics \eqref{dot_deltaxi}. A central result in \cite{LOHSLO} is that if there is a uniformly bounded metric $M: \rea^{n} \to \rea^{n \times n}_{\succ0}$, such that
\begin{equation}
\label{contraction}
\dot M + {\partial f \over \partial \bfx}^\top M + M{\partial f \over \partial \bfx} \preceq -2\lambda M \quad \mbox{(or $\prec 0$)},
\end{equation}
then the system \eqref{NLS} is contracting (or asymptotically contracting). The interested reader may refer to \cite{LOHSLO,FORSEP} for comprehensive introduction.

\section{Main Results on Constructive Solutions}
\label{sec:3}

In this section, we study \emph{constructive} solutions to Problem 1, as well as their convex representation. 

\subsection{Contracting reduced-order observers}
\label{sec:3a}

In this subsection, we will show that Problem 1 can be translated into contracting observers design, in which we aim to find mappings to achieve both \emph{contraction} and \emph{correctness}. Here, correctness refers to that the observer dynamics \eqref{observer:general} admits a particular solution $\xi_\star(t)$ such that the associated observer output $H(\xi_\star,y)$ always equals to the unknown system state $x(t)$. To see this, assuming that Problem 1 is solvable, thus any different solutions $\xi_a(t),\xi_b(t) \in \rea^{n_\xi}$ guaranteeing
\begin{equation*}
\label{contraction:xi}
\begin{aligned}
	& \limsup_{t\to\infty} \big|H(\xi_a(t),y(t)) - H(\xi_b(t),y(t))\big| \\
\le	& \limsup_{t\to\infty} \big| H(\xi_a(t),y(t)) - x(t)\big| + \big| H(\xi_b(t),y(t)) - x(t)\big|\\
= & 0.
\end{aligned}
\end{equation*}
Clearly, a sufficient condition to $ \big|H(\xi_a,y) - H(\xi_b,y)\big| \to 0$ is
contraction of the observer dynamics in \eqref{observer:general}. Furthermore, in order to guarantee \eqref{conv:general}, there should exist a \emph{particular solution} of \eqref{observer:general} which establishes the connection between $\xi$ and $x$. Intuitively, we construct, analyze, and implement an observer in a carefully selected $\xi$-coordinate, with $H(\cdot)$ linking these two coordinates. It is summarized as follows.

{\em Problem 2.} ({\em Contracting observer}) Find smooth mappings
$
N: \rea^{n_\xi}\times \rea^{n_y}\times \rea^{n_u} \to \rea^{n_\xi}$ and $ H:\rea^{n_\xi}\times \rea^{n_y} \to \rea^{n_x}
$
guaranteeing
\begin{itemize}
\item[\bf C1] (\emph{contraction}) the observer dynamics \eqref{observer:general} is contracting (or asymptotically contracting for the asymptotic case).
\item[\bf C2] (\emph{correctness}) the partial differential equation
\begin{equation}
\label{pde:correctness}
\begin{aligned}
{\partial H \over \partial \xi}(\xi,y)N(\xi,y,u)  + {\partial H\over \partial y}(\xi,y)f_y(H(\xi,y),y,u)
\\
=f_x(H(\xi,y),y,u)
\end{aligned}
\end{equation}
holds, and
$
\im_{\xi \in \cala}[H(\xi,y)]= \rea^{n_x}
$
uniformly in $y$ for some open set $\cala \subset \rea^{n_\xi}$, in which the observer dynamics \eqref{observer:general} is invariant.\footnote{That is $\xi(0)\in \cala$ implying $\xi(t)\in \cala$, $\forall t \ge 0$ for all $y \in \rea^{n_y}$.}\qed
\end{itemize}

\begin{proposition}
\label{prop:translation}\rm
Assume the system states of \eqref{eq:syst} are bounded. If we can find mappings $N(\xi,y,u)$ and $H(\xi,y)$ satisfying {\bf C1}-{\bf C2}, then the observer \eqref{observer:general} solves Problem 1. \qed
\end{proposition}
\begin{proof}
We only show the asymptotic case. For any $(u,y)$, invoking the invariance of \eqref{observer:general} in $\cala$ and the assumption that the image space of $H(\cdot,y)$ is the entire $\rea^{n_x}$, then $\forall x(0)\in\rea^{n_x}$ we can always find a point $\xi_\star(0)\in \cala$ such that
\begequ
\label{IC1}
H(\xi_\star(0),y(0))=x(0).
\endequ
Denote the solution of \eqref{observer:general} as $\xi_\star(t)$ from the initial condition $\xi_\star(0)$. The correctness \eqref{pde:correctness} guarantees 
$$
H(\xi_\star(t),y(t)) = x(t), \quad \forall t\ge 0,
$$
if the initial conditions satisfy \eqref{IC1}.

Since the system \eqref{observer:general} is asymptotically contracting and $\xi_\star(t)$ is a particular solution, we have that for any solution $\xi(t)$
\begequ\label{xi_conv}
\lim_{t\to\infty} |\xi(t) -\xi_\star(t)| =0.
\endequ

Invoking the state boundedness assumption, we, without loss of generality, denote $x \in \Omega_x \subset \rea^{n_x},~ y \in \Omega_y \subset \rea^{n_y}$ and $\xi_\star \in \Omega_\xi \subset \cala$ with $\im_{(\xi,y)\in \Omega_\xi\times \Omega_y}[H(\xi,y)] = \Omega_x$, and all these sets are defined bounded. There always exists a moment $T_1$ such that $\xi(t) \in B_\epsilon(\mbox{cl}(\Omega_\xi))$ for $t\ge T_1$ with a small $\epsilon>0$ due to the convergence \eqref{xi_conv} and $\xi_\star(t) \in \Omega_\xi$ for all $t\ge 0$. Then, for $t\in [T_1,\infty)$ we have
$$
\begin{aligned}
	\lim_{t\to\infty}|\hat x(t)-x(t)| & = \limsup_{t\to\infty}| H(\xi,y) - H(\xi_\star,y)|\\
	& \le \lim_{t\to\infty} \underset{\tiny \xi \in B_\epsilon(\mbox{cl}(\Omega_\xi))}{ \max}|\nabla_\xi H ||\xi(t)-\xi_\star(t) | \\
	& = 0,
\end{aligned}
$$
where we have used the state boundedness assumption, implying that $|\nabla_\xi H |$ is bounded in the closure of $\Omega_\xi$.
\end{proof}

In the sequel, we focus on the constructive solution and the convex representations to Problem 2, in this way solving Problem 1.

\subsection{General results for the case $n_\xi = n_x$}
\label{sec:3b}

We start with the basic case selecting the observer dimension $n_\xi = n_x$. In Problem 2, it is shown that we may consider this problem expressing the dynamics \eqref{eq:syst} in a transformed coordinate, and then study its contraction property by regarding $(u,y)$ as exogenous signals. Indeed, it is well known that the change of coordinate plays an important role in nonlinear observer design.

Let $\phi: \rea^{n_x} \times \rea^{n_y} \to \rea^{n_x}$ define the transformation
$$
(x,y) \mapsto (z,y) := (\phi(x,y),y).
$$
If we guarantee that $\phi(\cdot,y)$ is bijective from $\rea^{n_x}$ to $\rea^{n_x}$, then for each $y$ there is a left inverse $\phil: \rea^{n_x}\times \rea^{n_y} \to \rea^{n_x}$ such that
\begequ
\label{left-inv}
\phil(\phi(x,y), y) =x.
\endequ

Now the dynamics in the $z$-coordinate is written as 
\begin{equation}
\label{z}
\dot z = \beta(z,y,u)
\end{equation}
with
\begin{align}
\beta(z,y,u) & := f_z(\phil(z,y),y,u) \label{beta-fz}
\\
f_z(x,y,u) & ~ = \Phi_x(x,y) f_x(x,y,u)+ \Phi_y(x,y) f_y(x,y,u). \nonumber
\end{align}
and
\begin{equation}
\label{ExEy}
\Phi_x(x,y):= {\partial \phi \over \partial x}(x,y),~
\Phi_y(x,y):= {\partial \phi \over \partial y}(x,y).
\end{equation}
For convenience, we also use the notation $\Phi(\chi):= \nabla\phi^\top(\chi) = [\Phi_x(\chi), \Phi_y(\chi)]$. Supposing we are able to find a transformation $\phi(x,y)$ as described above such that the obtained system \eqref{z} is asymptotically contracting for any exogenous time-varying signals $(u,y)$, then it is trivial to solve Problem 2 by selecting $N(\cdot)=\beta(\cdot)$ and $H(\cdot) = \phil(\cdot)$. 

Based on the above intuitive idea, we are now in position to present the first main result of the paper on contracting reduced-order observers. Its convex representation is given in Section \ref{sec:3c}, and the extension to the case $n_\xi > n_x$ will be discussed in Section \ref{sec:3d}.

\begin{theorem}
\label{thm:reduced}\rm
Consider the system \eqref{eq:syst}, and assume that we can find mapping $\phi: \rea^{n_\chi} \to \rea^{n_x}$ and uniformly bounded Riemannian metric $M: \rea^{n_x} \to \rea^{n_x \times n_x}_{\succ 0}$ satisfying
\begin{itemize}
  \item[\bf A1] ({\em diffeomorphism}) the mapping $\phi_y(\cdot):=\phi(\cdot,y)$ defines a diffeomorphism uniformly in $y \in \rea^{n_y}$;
   \item[\bf A2] ({\em contraction})
  \begequ
  \label{cond:a2}
  \partial_{f_z} M \Big|_{\phi(\chi)}
  + M{\partial f_z \over \partial x}  \left[ {\partial \phi \over \partial x} \right]^{-1}
  +  \left[ {\partial \phi \over \partial x}\right]^{-\top} {\partial f_z \over \partial x}^\top M \prec 0
  \endequ
  (or $\preceq -2\lambda M,~\lambda>0$ for the exponential case) with
   $f_z(\cdot)$ defined in \eqref{beta-fz}, holds $\forall (\chi,u) \in \rea^{n_\chi} \times \mathcal{U}$.
\end{itemize}
Then, there exists an $n_x$-dimensional reduced-order observer guaranteeing \eqref{conv:general} globally.
\hfill$\triangleleft$
\end{theorem}
\begin{proof}
The assumption {\bf A1} defines a global diffeomorphism, and we thus, according to the global inverse theorem \cite{MIL}, obtain that
\begequ
\label{inverse}
{\partial \phi \over \partial x}(x,y)
=
\left[{\partial \phil\over \partial z}(z,y)\right]^{-1}_{z=\phi(\chi)}.
\endequ

The dynamics of $z$ can be written as \eqref{z}. Now design an observer \eqref{observer:general} by selecting
\begin{equation}
\label{observer:mapping}
\begin{aligned}
N(\xi,y,u)  = f_z(\phil(\xi,y),y,u), ~
H(\xi,y)  = \phil(\xi,y).
\end{aligned}
\end{equation}
The correctness condition {\bf C2} can be verified trivially from the above construction. Therefore, $z(t)=\phi(x(t),y(t))$ is a particular solution of the observer dynamics from $z(0)$.

Since $\phi(\cdot,y)$ is a diffeomorphism for any $y$ due to {\bf A1}, in the sequel we only need to show the convergence of $\xi$ to $z$. The dynamics of them are given by
$$
\dot \xi = N(\xi,y,u), \quad \dot z= N(z,y,u),
$$
where $(u,y)$ are viewed as (available) exogenous signals. Towards this end, we first need to calculate the differential characterization along the solution of \eqref{observer:general} with $y(t)$ governed by the system \eqref{eq:syst}. Note that the solution $y(t)= Y(t;\chi_0,u)$ is an exogenous signal, which is independent of $\xi(t)$. Consider the initial conditions $z(0)$ and $\xi(0)$ are connected by a smooth curve $\gamma:[0,1]\to \rea^{n_x}$, such that $\gamma(0)=z(0)$ and $\gamma(1)=\xi(0)$. The solution from $\gamma(s)$, $s\in[0,1]$ under the exogenous signals $(u,y)$ is denoted as $\xi(t;\gamma(s),y,u)$. Thus, we have the differential dynamics
$$
\begin{aligned}
 {d\over dt} \delta\xi & = {\partial N \over \partial \xi}(\xi,y,u) \delta \xi
 \\
 & = {\partial f_z \over \partial \xi}(\phil(\xi,y),y,u)\delta \xi
 \\
 & \overset{}{=}~ {\partial f_z \over \partial x}(\phil(\xi,y),y,u) {\partial \phil \over \partial z}(\xi,y)  \delta\xi \\
  &  \overset{\eqref{inverse}}{=} {\partial f_z \over \partial x}(\hat x,y,u)   \left[ {\partial \phi \over \partial x}(\hat x,y) \right]^{-1} \delta \xi,
\end{aligned}
$$
with the infinitesimal displacement $\delta\xi = {\partial \xi(t;\gamma(s),y,u) \over \partial s}$.

Define a differential Lyapunov function as
$
V(\xi,\delta \xi) =\delta \xi^\top  M(\xi) \delta \xi,
$
yielding
$$
\begin{aligned}
\dot{\aoverbrace[L1R]{  V(\xi,\delta\xi)}{}}
 & =
 \delta \xi^\top \left(
 \partial_{f_z} M
  + 2{\tt sym}\left(M{\partial f_z \over \partial x}  \left( {\partial \phi \over \partial x} \right)^{-1} \right)  \right)\delta\xi
  \\
& \overset{\eqref{cond:a2}}{<} 0
\end{aligned}
$$
by replacing $x$ with $\hat x= \phil(\xi,y)$, and thus we have
$
\lim_{t\to\infty} V(\xi(t),\delta\xi(t)) = 0.
$
Therefore, we obtain
$$
\begin{aligned}
 &\limsup_{t\to\infty}|\xi(t) - z(t)|
\\
&\le  \lim_{t\to\infty} c\int_0^1 V(\xi(t;\gamma(s),y,u),\delta \xi(t;\gamma(s),y,u) )ds 
=0,
\end{aligned}
$$
for some $c>0$. It completes the proof.
\end{proof}

\begin{remark}
We adopt in the assumption {\bf A2} a Riemannian metric $\delta\xi^\top M(\xi)\delta \xi$ to design an observer, which is very general for systems living in Euclidean space under some smoothness assumption. It can be generalized using a Finsler metric $F(\xi,\delta\xi)$ \cite{FORSEP}, then guaranteeing the induced distance monotonically decreasing. It is particularly useful if system states live in Riemannian manifolds, {\em e.g.}, matrix Lie groups.
\end{remark}

\begin{remark}
It is underlined here that the sufficient condition given in Theorem \ref{thm:reduced} only relies on the existence of the mapping $\phi(x,y)$, though the left inverse $\phil$ is needed when realizing the design. Assuming that we have already found $\phi(x,y)$, its left inverse is given by the solution of the optimization
$
\phil(\xi,y) = \mbox{argmin}_{\hat x \in \rea^{n_x}}  | \phi(\hat x,y) -\xi |^2,
$
which may be implemented by off-line searching for analytic solutions, or on-line searching for numerical results, see \cite{KREENG} for an explicit formula.
\end{remark}

In the above analysis, since we have used the inverse function theorem, the correctness condition {\bf C2} is satisfied automatically. Some further discussions about Theorem \ref{thm:reduced} will be given in Section \ref{sec4}.

\subsection{Convex sets of nonlinear contracting observers}
\label{sec:3c}


The sufficient condition---Assumptions {\bf A1} and {\bf A2}---in Theorem \ref{thm:reduced} is clearly non-convex in general, since it depends in a nonlinear way on the left inverse of the unknown function $\phi(x,y)$. In this subsection, we will present some \emph{convex representations}, which are motivated by, and extend to the reduced-order case, our previous results in \cite{MAN}. Indeed, searching for $\phi(\cdot)$ and $M(\cdot)$ can be viewed as an optimization problem on an infinite-dimensional function space. If we could obtain convex sufficient conditions, one may verify these differential equations and inequalities using the sum-of-square relaxation, thus making the observer design numerically tractable.

The condition {\bf A1} needs to check the non-convex function $\det\{\nabla_x \phi\}$. We have its convex approximation below.

\begin{lemma}
\label{lem:a1-cv}
	The convex condition {\bf H1} is sufficient to {\bf A1}.
\begin{itemize}
\item[{\bf H1}] ({\em the monotone condition}) For any $y \in \rea^{n_y}$,
  \begequ
  \label{ineq-monotone}
   {\partial \phi \over \partial x}(x,y) + \left[ {\partial \phi \over \partial x}(x,y) \right]^\top \succeq k I_{n_x} , \quad k>0.
   \endequ
\end{itemize}
\end{lemma}
\vspace{0.3cm}
\begin{proof}
Since the partial derivative operator is linear, the inequality \eqref{ineq-monotone} is convex in $\phi(x,y)$. Following \cite[Theorem 1]{TOBetal} or the Hadamard's global inverse theorem \cite[Theorem 2]{MIL}, we can verify the diffeomorphism of $\phi_y(\cdot)=\phi(\cdot,y)$.
\end{proof}

Apart from the convexity of {\bf H1}, another benefit of the monotonicity we imposed relies on the existing efficient algorithms to calculate the inverse mapping $\phil$ when implementing the observer, {\em e.g.}, the Newton method and the ellipsoid method \cite{TOBetal,NEM}.

Now we move forward to the condition {\bf A2} and study its convex representation. Note that the contraction analysis was done in the transformed coordinate $z=\phi(\chi)$, in which the dynamics is given in \eqref{beta-fz}. Its second equation $f_z(\chi,u) = \Phi(\chi) f(\chi,u)$ shows that $f_z(\cdot)$ is in the ``image space'' of a given vector field $f(\chi,u)$. An important point is that the relationship between $\phi(\cdot)$ and $f_z(\cdot)$ is linear. If we regard $f_z(\cdot)$ as a \emph{new decision variable}, and jointly search for the transformation $\phi(\chi)$ and the mapping $f_z(\chi,u)$, we may be able to obtain a convex set of reduced-order contracting observers. 

In terms of the new decision variable $f_z(\cdot)$ involved, we additionally---compared with {\bf A2}---require the following correctness condition, which is a variant of {\bf C2} in our design framework. It is a linear equality constraint with respect to $\phi(\chi)$ and $f_z(\chi,u)$, thus convex.

\begin{itemize}
    \item[\bf H2] ({\em the modified correctness})
\begin{equation}
\label{correctness}
\Phi_x(\chi)f_x(\chi,u) + \Phi_y(\chi) f_y(\chi,u) = f_z(\chi,u)
\end{equation}
for all $(\chi,u) \in \rea^{n_\chi} \times \rea^{n_u}$.
\end{itemize}

Then, regarding the convex representation of the contraction condition {\bf A2}, we have the following.

\begin{lemma}
\label{lem:a2-cvx}
Under {\bf H1}, the following convex conditions {\bf H3} or {\bf H4}, together with {\bf H2}, imply the condition {\bf A2}, {\em i.e.},
$$
\Big[~\mbox{\bf H2}\wedge (\mbox{\bf H3} \vee\mbox{\bf H4})~\Big]
\quad \implies \quad 
\mbox{\bf A2}.
$$
\begin{itemize}
\item[\bf H3] By restricting the transformation
\begin{equation}
\label{phi:part_l}
\phi(x,y) = Px +  \varphi(y)
\end{equation}
with a smooth mapping $\varphi: \rea^{n_y} \to \rea^{n_x}$, then search for $P\succ 0$, $\varphi(\cdot)$, $f_z(\cdot)$ and $Q$ verifying
\begin{equation}
\label{ct1}
F(x,y,u) + F^\top(x,y,u) + Q \preceq 0,
\end{equation}
with
$
F(x,y,u) := {\partial f_z \over \partial x}( x,y,u),
$ and $Q \succ 0$ (or $Q=2\lambda P$ for the exponential case with $\lambda>0$).
\item[\bf H4] Find $\phi(x,y)$, $Q$, $f_z(\cdot)$, metric $P\succ 0$ and $r>0$ verifying
\begin{equation}
\label{ct2}
\begmat{2\mathtt{sym}(\Phi_x-{r\over2}F) -P-rQ & \Phi_x + {r\over2}F \\ (\Phi_x+{r\over2}F)^\top & P}
\succeq 0,
\end{equation}
with $Q \succ 0$ (or $Q=2\lambda \Phi_x^\top P^{-1} \Phi_x$ for contraction). \qed
\end{itemize}
\end{lemma}
\begin{proof}
The conditions \eqref{ct1} and \eqref{ct2} are state-dependent linear matrix inequalities (LMIs), which are jointly convex in $\phi$ (or $\varphi$), $f_z,P$ and quasi-convex in $\lambda$ for the exponential case. In {\bf A2}, the function $f_z$ is defined as \eqref{beta-fz}. In the convex representation, we regard $f_z$ as a free mapping to be searched for, which, thus, should satisfy the correctness condition {\bf H2}.

The remainder of the proof is to verify the inequality in {\bf A2}. First, we consider the assumption {\bf H3} with the parameterization $\phi(x,y)=Px+\varphi(y)$, thus
$$
\Phi_x = {\partial \phi\over\partial x}(x,y) = P.
$$
Then, it is straightforward to get\footnote{We give the details of the exponential case, and the asymptotic case follows in a similar way.}
$$
\begin{aligned}
  \eqref{ct1} & ~ \iff F+F^\top \preceq - 2\lambda \Phi_x^\top P^{-1} \Phi_x
  \\
  &~ \iff P^{-1} F \Phi_x^{-1} + \Phi_x^{-\top} F^\top P^{-1} \preceq -2\lambda M
  \\
  & ~\iff \partial_{f_z} P^{-1} + P^{-1} F \Phi_x^{-1} + \Phi_x^{-\top} F^\top P^{-1} \preceq -2\lambda M
  \\
  &\overset{M:= P^{-1}}{\iff} \mbox{\bf A2},
\end{aligned}
$$
verifying the first convex sufficient condition in {\bf H3}.

For the case in {\bf H4}, the inequality \eqref{ct2} is equivalent to
\begin{equation}
\label{ct2-2}
\begin{aligned}
 \left(\Phi_x - {r\over 2} F \right) + \left(\Phi_x - {r\over 2} F\right)^\top
 - \left|\Phi_x + {r\over 2} F\right|^2_{ P^{-1}} \\
 - P - rQ \succeq 0.
\end{aligned}
\end{equation}
As show in \cite[Section III]{TOBetal}, the inequality \eqref{ct2-2} is a sufficient condition to 
\begin{equation}
\label{contraction:implicit}
\begin{aligned}
\Phi_x^\top P^{-1} F + F^\top P^{-1} \Phi_x
\preceq -  Q .
\end{aligned}
\end{equation}
Namely, using the inequalities
$$
\Phi^\top_x M F + F^\top M \Phi_x
=
{1\over r}[ |\Phi_x +{r\over2}F|^2_{M} - |\Phi_x -{r\over2}F|_{M}^2]
$$
and
$
-c^\top P^{-1}c \le b^\top P b - 2b^\top c
$
with $M=P^{-1}$ successively, we get \eqref{contraction:implicit} from \eqref{ct2-2}. It yields
$$
\begin{aligned}
 \eqref{contraction:implicit} & ~\iff
 \Phi_x^\top P^{-1} F + F^\top P^{-1} \Phi_x
\preceq -  2\lambda \Phi_x^\top P^{-1} \Phi_x
\\
& ~ \iff
  P^{-1} F \Phi_x^{-1} + \Phi_x^{-\top} F^\top P^{-1} 
\preceq -  2\lambda  P^{-1}
\\
& ~\iff \partial_{f_z} P^{-1} + P^{-1} F \Phi_x^{-1} + \Phi_x^{-\top} F^\top P^{-1} \preceq -2\lambda M
  \\
  &\overset{M:= P^{-1}}{\iff} \mbox{\bf A2},
\end{aligned}
$$
where in the second equivalence we have used the assumption {\bf H1} to guarantee the existence of the inverse matrix of $\Phi_x$. It completes the proof.
\end{proof}

\begin{remark}\rm
We will show, via examples, that the convex condition {\bf H3} is quite general to cover several physical models. It enjoys another merit of simple left inverse mapping $\phil(\xi,y) = P^{-1}[\xi -\varphi(y)]$. Besides, the condition {\bf H3} can be generalized into the form 
$
\phi(x,y)=P(y)x + \varphi(y),
$
with $P(\cdot)$ dependent on $y$. For this case, \eqref{ct1} becomes the convex constraint
$
F(x,y,u) + F^\top(x,y,u) + \dot P(y) + Q \preceq  0.
$
\end{remark}

To accomplish the reduced-order observer design, we would like to search $\phi(\chi)$ and $f_z(\chi,u)$ simultaneously over finite dimensional space. Thus, we parameterize them as
\begequ
\label{parameterization}
\phi(\chi) = \sum_{i=1}^{q_1} \theta_i \phi_i(\chi),
\quad
f_z(\chi,u) = \sum_{i=q_1+1}^{q_2} \theta_i f_{z,i}(\chi,u)
\endequ
with basis mappings $\phi_i$ and $f_{z,i}$ ($i=1,\ldots,q_2$) smooth enough, and $q_1,q_2 \in \mathbb{N}_+$ verifying $q_1 <q_2$. The role of parameterization is to reformulate the obtained infinite dimensional optimization over function space into a finite dimensional convex problem, which is efficiently solvable.

We are in position to present the second main result of the paper about a convex set of contracting reduced-order observers, which is a \emph{constructive} solution to Problem 2.

\begin{theorem}
\label{thm:convex-observer}\rm
Given the nonlinear system \eqref{eq:syst}, for the case $n_\xi =n_x$ a convex set of contracting reduced-order observers is given by \eqref{parameterization} together with the observer dynamics
\begin{equation}
\label{convex_set_ob}
\begin{aligned}
\dot{\xi}  = f_z(\hat x, y,u), \quad
\hat x  = \phil(\xi, y)
\end{aligned}
\end{equation}
with $\phil(\cdot,y)$ the left inverse of $\phi(\cdot,y)$ solving Problem 2, where $\theta \in \Theta$ is defined by the convex constraints {\bf H1}, {\bf H2}, and either {\bf H3} or {\bf H4}.
\end{theorem}
\begin{proof}
As already shown above, the condition {\bf H1} guarantees the existence of inverse mapping $\phil(\cdot)$ globally. Following Theorem \ref{thm:reduced} and Lemmata \ref{lem:a1-cv}-\ref{lem:a2-cvx}, we get the claim.
\end{proof}

\subsection{Extension to the immersion case $n_\xi >n_x$}
\label{sec:3d}

In the previous subsections, we considered the case that the observer dynamics shares the dimension of the systems state. It, however, is openly recognized that by considering injective immersion may allow to take into account a larger class of nonlinear control systems, being widely used in observer design with various purposes.

In this subsection, we extend our results to the case $n_\xi > n_x$. The motivation is threefold.
\begin{itemize}
    \item Sometimes there exists observable singularity for the system to be estimated. By increasing observer dimensions, we may be able to guarantee the injectivity of changes of coordinates, see \cite{RAPMAL,ANDPRA} and \cite[Chapter 4]{BES}. 
    \item New mappings will appear in \eqref{correctness}, the new degree of freedom from which relaxes the correctness condition.
    \item For a given nonlinear system, it may provide a strictly larger set, compared to the case of $n_\xi = n_x$, of convergent observers, among which we may select an optimal design in some sense, in this way enhancing performance.\footnote{A similar problem---parameterizing all convergent observers for linear time-invariant systems---was comprehensively studied in \cite{GOOMID}. }
\end{itemize}

Following the idea in Section \ref{sec:3a}, we would like to find $\phi: \rea^{n_x} \times \rea^{n_y} \to \rea^{n_\xi}$ with $n_\xi > n_x$. However, two difficulties arise, namely, (i) $\nabla_x \phi(x,y)$ is not square, thus the inverse function theorem not applicable; and (ii) the image space $\cali_\xi$ of $\phi(\cdot,y)$ may be only an \emph{open subset} of $\rea^{n_\xi}$, and thus the left inverse $\phil(\cdot,y)$ is not defined globally. A possible approach to deal with this issue is to construct an extended left mapping $\phile: \rea^{n_\xi} \times \rea^{n_y}$ satisfying
$
\phile(z,y) = \phil(z,y), ~\forall (z,y) \in \cali_\xi\times \rea^{n_y}.
$
It is usually applied to the case in which the transformation $\phi(\cdot)$ has been \emph{already known}. However, it is not suitable in our situation, since our target is to find the mapping $\phi(\cdot)$.

In order to be able to keep consistent with the proposed constructive method, as well as to circumvent the difficulties above, we introduce an \emph{augmentation} of the $x$-coordinate into the extended system state $(x,w)$ given by
\begin{equation}
\label{dotw}
    \dot w = f_w (x,w,y,u)
\end{equation}
with $w \in \rea^{n_w}$, $n_w=n_\xi - n_x$ and $x_e := \col(x,w)$.\footnote{In \cite{BERetalsiam}, complementary operation of full column-rank Jacobian is used to express an observer in the preferred coordinate assuming that an observer \emph{exists already}.} The following lemma indicates that the augmented state $w$ should not change the ``detectability'' of the system \eqref{eq:syst}.
\begin{lemma}
\label{lem:argument}\rm
Assume that there exists an exponential observer for the system \eqref{eq:syst}.
\begin{itemize}
\item[1)] If the augmented system \eqref{dotw} is contracting, then there exists an observer for the extended state $(x,w)$.
\item[2)] Conversely, for the case that the given system is autonomous ({\em i.e.} $u=0$) with all states bounded, assuming that the $(x,w)$-system has an exponential observer \eqref{convex-observer:argu} with a smooth Lyapunov function $V(x_e,y,\xi)$ such that
$$
\begin{aligned}
a_1|\xi - \phi(x_e,y)|^2 \le V(x_e,y,\xi) & \le a_2 |\xi - \phi(x_e,y)|^2 \\
\dot{\aoverbrace[L1R]{ V(x_e,y,\xi)}}  & \le -\lambda V(x_e,y,\xi)
\end{aligned}
$$
with $a_1,a_2,\lambda>0$, then the system \eqref{dotw} is contracting.
\end{itemize}

\end{lemma}
\begin{proof}
For the first claim, it can be proved by cascading the existing observer with 
$$
\dot{\hat w} = f_w(\hat{x},\hat w,y,u)
$$
with $\hat x$ the estimate of $x$ from the existing observer, and invoking the solution continuity of differential equations, and the state boundedness.

Regarding the second claim, we consider the autonomous case $u=0$. According to \cite[Proposition 4]{ANDetal}, for any $(x_e,y,\delta x_e,\delta y) \in \rea^{n_\xi+n_y} \times T \rea^{n_\xi+n_y} $, the norm of infinitesimal displacement $|\col(\delta x_e,\delta y)|$ is monotonically decaying along trajectories, tangentially to some output function level sets. That is, in our case if we fix the differential output $\delta y(t)\equiv 0$, then the obtained LTV system
\begin{equation}
\label{ltv:xw}
\left(\begin{aligned} \delta\dot x\\ \\ \delta \dot w \end{aligned}\right)
= 
\left(\begin{aligned}
& {\partial f_x \over \partial x}(x,y,0) & & 0 \\
& {\partial f_w \over \partial x }(x,w,y,0) & & {\partial f_w \over \partial w}(x,w,y,0)
\end{aligned}\right)
\left(\begin{aligned} \delta  x\\ \\ \delta   w \end{aligned}\right)
\end{equation}
satisfies
$$
\left|\begmat{\delta x(t) \\ \delta w(t)}\right|^2 \le a_0 \exp(-\lambda t)\left| {\partial \phi_e \over \partial x_e}(x_e(0),y(0)) \begmat{\delta x(0) \\ \delta w(0)}\right|^2,
$$
for some $a_0>0$. The inequality holds for any initial conditions, and thus we consider a particular one $\delta x(0)=0$, for which we have 
$
\delta x(t) =0, ~\forall t \ge 0
$
along \eqref{ltv:xw}. Then, it yields
$$
\delta \dot w = {\partial f_w\over \partial w}(x,w,y,0)\delta w
$$
admitting
$$
|\delta w(t)|^2 \le a_0\exp(-\lambda t)\left|{\partial \phi_e(x_e(0),y(0)) \over \partial w} \delta w(0)\right|^2
$$
uniformly along any trajectories. It implies that the augmentation dynamics \eqref{dotw} should be contracting.
\end{proof}

The second claim shows the necessity to impose contraction properties with respect to $w$ when introducing the augmentation \eqref{dotw}, a sufficient condition for which is to find a uniformly bounded metric $M_w: \rea^{n_w} \to \rea^{n_w\times n_w}$ satisfying
\begin{equation}
\label{contraction:w}
\partial_{f_w} M_w + M_w{\partial f_w \over\partial w} + {\partial f_w \over\partial w}^\top M_w \preceq - 2\lambda_w M_w, \; \lambda_w>0.
\end{equation}
Indeed, we can always find a vector field $f_w$ satisfying \eqref{contraction:w}, the simplest one among which may refer to
\begequ
\label{fwA}
f_w = Aw + f_a(x,y,u)
\endequ
with Hurwitz matrix $A$ and smooth function $f_a$. 

Now, it is straightforward to extend Theorems \ref{thm:reduced}-\ref{thm:convex-observer} to this case. We only give here the convex variant of the latter. Similarly, we may parameterize mappings $f_{ze}: \rea^{n_\xi +n_y + n_u} \to \rea^{n_\xi}$, $\phi : \rea^{n_\xi +n_y + n_u }\to \rea^{n_\xi}$ and $f_w :\rea^{n_\xi +n_y + n_u} \to \rea^{n_w}$, and matrices $P_e \in \rea^{n_\xi \times n_\xi}_{\succ 0}$ and $M_{w} \in \rea^{n_w \times n_w}_{\succ 0}$ as
\begin{equation}
\label{parameterization:argu}
\begin{aligned}
 f_{ze}& = \sum_{i=q_1+1}^{q_2}  \theta_i f_{ze,i} (x_e,y,u)
 ,\quad
  \phi  = \sum_{i=1}^{q_1} \theta_i \phi_i(x_e,y)
 \\
 f_w & = \sum_{i=q_2+1}^{q_3}  \theta_i f_{w,i} (x_e,y,u)
\end{aligned}
\end{equation}
for some $q_i\in \mathbb{N}_+$ and $q_i>q_j~(i>j)$. 
For convenience, we also define the gradients
$$
F_e(x_e,y,u) = {\partial f_{ze} \over \partial x_e}(x_e,y,u),
\quad
\Phi_{xe}(x_e,y) = {\partial \phi \over \partial x_{e}}(x_e,y).
$$

\begin{proposition}
\label{prop:convex-observer-immersion}\rm
Consider the system \eqref{eq:syst} with a fixed augmentation vector field $f_w$ satisfying \eqref{contraction:w}. For the case $n_\xi >n_x$ a convex set of contracting reduced-order observers is given by \eqref{parameterization:argu}, together with the observer dynamics 
\begequ
\label{convex-observer:argu}
\begin{aligned}
\dot \xi  = f_{ze}(\hat x, \hat w, y,u) ,\quad
\begmat{\hat x\\ \hat w}  = \phil(\xi,y),
\end{aligned}
\endequ
with $\phil(\cdot,y)$ the left inverse of $\phi(x_e,y)$ and $\theta \in \Theta$ defined by the convex constraints {\bf H1$'$}, {\bf H2$'$} and either {\bf H3$'$} or {\bf H4$'$}.
\begin{itemize}
    \item[\bf H1$'$] ({\em monotonicity}) $\forall y \in \rea^{n_y}$
$$
   {\partial \phi \over \partial x_e}(x,w,y) + \left[ {\partial \phi \over \partial x_e}(x,w,y) \right]^\top \succeq k I_{n_\xi} , \quad k>0.
$$
    \item[\bf H2$'$] ({\em correctness})
\begin{equation}
\label{correctness:aug}
\begin{aligned}
\Phi_x(x_e,y)f_x(\chi,u) 
&+ \Phi_y(x_e,y)f_y(\chi,u) \\
&+ \Phi_w(x_e,y) f_w(x_e,y,u)
=
f_{ze}(x_e,y,u)
\end{aligned}
\end{equation}
with
$
\Phi_w = \nabla_w \phi^\top(x,w,y,u).
$
\item[\bf H3$'$] By restricting 
$
\phi(x,w,y) = P_e\col(x,w) +  \varphi(y)
$
with $\varphi: \rea^{n_y} \to \rea^{n_\xi}$, and requiring
\begin{equation}
\label{ct1e}
F_e(x_e,y,u) + F_e^\top(x_e,y,u) + Q_e \preceq 0,
\end{equation}
with $P_e \succ 0$, and $Q_e \succ 0$ (or $Q_e=2\lambda P_e$ for the exponential case with $\lambda >0$).
\item[\bf H4$'$] The stronger contraction condition 
$$
\begin{aligned}
 2{\tt sym}(\Phi_{xe} - {r\over 2} F_e)   - \big|\Phi_{xe} + {r\over 2} F_e\big|^2_{P_e^{-1}} - P_e - rQ_e  \succeq 0,
\end{aligned}
$$
with $r>0$, and $Q_e \succ 0$ (or $Q_e=2\lambda \Phi_{xe}^\top P_e^{-1} \Phi_{xe}$ for contraction with rate $\lambda>0$). \qed
\end{itemize}
\end{proposition}

\begin{proof}
The proof follows \emph{mutatis mutandis} the one of Theorem \ref{thm:convex-observer} and the analysis above.
\end{proof}

In the above proposition, we fix the vector field $f_w$ and then search for mappings $\phi$ and $f_{ze}$, since \emph{bi-linear} operations appear in both the correctness condition \eqref{correctness:aug} and the contraction condition \eqref{contraction:w} if we regard both $f_w$ and $M_w$ as decision variables, making the optimization problem non-convex. Then, we get a state-dependent bilinear matrix inequality (BMI) problem. An alternative to maintain convexity is to fix $\Phi_w$, then providing the degree of freedom to search for the vector field $f_w$. It is clear that \eqref{correctness:aug} relaxes the correctness condition \eqref{correctness} by involving additional terms. On the other hand, the excessive coordinates may provide more observer candidates, thus making it possible to achieve performance enhancement by carefully selecting among them.

\begin{remark}
If the augmented $w$-dynamics is in the form of \eqref{fwA} with $f_a$ only dependent on $(u,y)$, the proposed coordinate change is similar to the so-called \emph{filtered transformation} introduced in \cite{MARTOM}. Such a technique has been proved very useful in \emph{adaptive observer} design, {\em i.e.}, estimating unknown states under parameter uncertainty.
\end{remark}

\section{Further Results and Discussions}
\label{sec4}

In this section, we give further discussions about the main results in Section \ref{sec:3}, and show the connections with some existing results in the literature. Some remarks are in order.

\begin{remark}
For the case $n_\xi=n_x$, it is sometimes unnecessary to require the transformation $\phi(\cdot,y)$ a bijection from $\rea^{n_x} \to \rea^{n_x}$. A possible case is that the image space is an open subset in $\rea^{n_x}$, {\em i.e.}, $\phi(\rea^{n_x},y) \subset \rea^{n_x}$. For such a case, image extension or projection operator may be useful to continue design \cite{BERetalsiam,RAPMAL}. However, we adopt {\bf H1}, on one hand, to circumvent the possible difficulty from non-subjectivity, and one the other hand, to obtain a \emph{convex} condition with computational consideration.
\end{remark}

    %

\begin{remark}
\label{remark:doty}\rm
The obtained reduced-order observer is given in the $z$-coordinate, but the correctness can also be verified in the original coordinate. We may write the observer \eqref{convex_set_ob} in the $x$-coordinate, though \emph{unimplementable}, which is given by
\begin{equation}\label{dothatx}
\begin{aligned}
  \dot{\hat x}  &~ = {\partial \phil \over \partial z}(\xi,y) \blue{f_z} \big(\phil(\xi, y),y, u\big)
  + {\partial \phil \over \partial y}(\xi,y) \dot y
  \\
  &~ = ~ \left[ {\partial \phi \over \partial x}(\hat x,y) \right]^{-1} \blue{f_z} \big(\hat x, y, u\big) +
   {\partial \phil \over \partial y}(\xi,y) \dot y \\
  &~ =~ \left[ {\partial \phi \over \partial x}(\hat x,y) \right]^{-1}
  \Bigg(
  {\partial \phi \over \partial x }(\hat x,y) f_x(\hat x,y,u) \\
  & \hspace{3cm}+ {\partial \phi \over \partial y }(\hat x,y) f_y(\hat x,y,u)
  \Bigg) + {\partial \phil \over \partial y}(\xi,y) \dot y
  \\
  & ~ = ~f_x(\hat x,y,u) + [\Phi_x^{-1}\Phi_y ] \Big( f_y(\hat x,y,u) - f_y(x,y,u)
  \Big),
\end{aligned}
\end{equation}
where in the last equation we have used the fact that
\begequ
\label{eq:T}
\begmat{x\\y }
\mapsto
T(x,y):= \begmat{\phi(x,y) \\ y}
\endequ
is another diffeomorphism, and thus using the inverse function theorem again yields
$$
{\partial \phil \over \partial y}(z,y)\Big|_{z=\phi(x,y)} = - \left[{\partial \phi \over \partial x}(x,y)\right]^{-1} {\partial \phi \over \partial y}(x,y).
$$
The last line of \eqref{dothatx} verifies the invariance of the proposed observer, {\em i.e.},
$
\hat x(t) \equiv x(t),~\forall t\geq 0,~\mbox{if}~ x(0)=\hat x(0),
$
since the last term $(f_y(\hat x,y,u) - f_y(x,y,u))$ plays the role of ``updated (or innovation) term''. Note that, however, the observer dynamics \eqref{dothatx} cannot be realized due to the unavailability of $\dot y$, equivalently $f_y(x,y,u)$, and we should implement it in the $z$-coordinate instead.
\end{remark}

We next show the connections among the proposed designs and some existing results.

\subsection{Linear time-invariant systems}

It is easy to verify that the conditions either in Theorem \ref{thm:reduced} or Theorem \ref{thm:convex-observer} ({\bf H3}) are equivalent to the detectability of linear time-invariant (LTI) systems, summarized below.

\begin{proposition}\rm
Consider the LTI system
\begin{equation}\label{lti}
\begmat{\dot x \\ \dot  y} = \underbrace{\begmat{A_{11} & A_{12} \\ A_{21} & A_{22}}}_{:=A} \begmat{x \\ y} + \begmat{B_1 \\ B_2 }u
\end{equation}
with unknown state $x\in \rea^{n_x}$, input $u\in \rea^{n_u}$ and output $y \in \rea^{n_y}$. The conditions {\bf A1}-{\bf A2} are the necessary and sufficient condition to the detectability of the system \eqref{lti}.
\qed
\end{proposition}
\begin{proof}
For linear systems, we consider linear transformation in the form of
$$
\phi(x,y) = T_1 x + T_2y
$$
for some matrices $T_1$ and $T_2$, with $T_1$ full rank. The sufficient part is trivial, since {\bf A1}-{\bf A2} imply the existence of a linear Luenberger observer\footnote{Here, we refer to Luenberger's initial idea by involving a change of coordinate in \cite{LUE}. We adopt the qualifier ``initial'' since what is nowadays called ``Luenberger observer'' is different from \cite{LUE}, see \cite{BER} for more discussions.}
$$
\dot z = Fz + Dy + Eu, \quad \hat x = T_1^{-1}(z- T_2y)
$$
for some matrices $D,E$, and $F$ is Hurwitz. It is equivalent to detectability in the LTI context, see for example \cite[Section III]{LUEtac71}. Regarding the necessary part, we have
$$
\begin{aligned}
& \mbox{Detectability of } \eqref{lti}  \quad \\
& \quad  \Longleftrightarrow \quad  \exists ~L \in \rea^{(n_x +n_y) \times n_y} \quad \mbox {s.t.} ~  \lambda_i\{ A + L [0 ~~ I_{n_y}]\} \in \mathbb{C}_- \\
& \quad  \Longleftrightarrow \quad  \exists ~\tilde L \in \rea^{n_x \times n_y} \quad \mbox {s.t.} ~  \lambda_i\{A_{11} + \tilde L  A_{12}\} \in \mathbb{C}_-
\end{aligned}
$$
where in the second equivalence we have used \cite[Lemma 1]{BESHAM}, and $\lambda_i\{\cdot\}$s denote eigenvalues of square matrices. Now we select the mapping
$$
\phi(x,y) = x+ \tilde{L} y
$$
with $\tilde L$ guaranteeing $A_{11} + \tilde L  A_{12}$  a Hurwitz matrix. It is straightforward to verify that $\phi(\cdot)$ satisfies {\bf A1}-{\bf A2}.
\end{proof}

Indeed, the obtained reduced-order observer for LTI systems is precisely the observer proposed in the pioneering paper \cite[Theorem 4, pp. 77]{LUE}. 

\subsection{Strongly differentially observable systems}

In this subsection, we show the set of systems identified by Theorem \ref{thm:reduced} is larger than the set of strongly differentially observable systems. It is well known that high-gain observers are applicable to this class of nonlinear systems. For simplicity, we consider the single-output ($n_y=1$) autonomous system
\begin{equation}
\label{single-output}
 \dot x = f_x (x,y), \quad \dot y = f_y(x,y).
\end{equation}

\begin{definition}
\label{def:sdo}\rm
({\em Strong differential observability} \cite{GAUKUPbook}) The single-output system \eqref{single-output} is strongly differentially observable, if there exists $n_\ell \in \mathbb{N}$ such that 
$
{\calh}^e_{\ell}(\chi) = \begpma{y, \calh_{\ell}^\top(\chi) }^\top
$
is an injective immersion to $\chi= \col(x,y)$, where
$
\calh_\ell(\chi) = [f_y(\chi), L_f f_y(\chi), \ldots, L_f^{n_\ell -2}f_y(\chi) ]^\top.
$
\end{definition}

The above definition can be equivalently expressed as $\calh_l(\chi)$ is an injective immersion to $x$ uniformly in $y$. We have the following.

\begin{proposition}
\label{prop:sdo}\rm
If the system \eqref{single-output} is strongly differentially observable with $n_\ell = n_x+n_y$, then the function
$$
\phi(x,y) = \calh_\ell (x,y) - \ell\Lambda y
$$
satisfies {\bf A2} semi-globally, where $\ell>0$, $\Lambda = [\lambda_1,\ldots, \lambda_{n_x}]^\top $ with all the roots of $s^{n_x} + \lambda_1 s^{n_x-1} + \ldots \lambda_{n_x}$ in $\mathbb{C}_-$.
\end{proposition}
\begin{proof}
We may verify that the above construction yields $\dot z = f_z(x,y)$ with
$$
f_z = \ell Q\big(\calh_\ell(x,y) - \Lambda y \big) + \ell Q\Lambda y + b(x,y)
$$
and
$$
Q = \begmat{-\lambda_1 & 1 & & \\ \vdots & & \ddots & \\ \vdots & & & 1 \\ -\lambda_{n_x} & 0 &\ldots & 0}
, \quad
b(x,y) = \begmat{0 \\ \vdots \\ 0 \\ L_f^{n_x-1}f_y(x,y)}.
$$
We note that
$$
\begin{aligned}
{\partial f_z \over \partial x}  \left[ {\partial \phi \over \partial x} \right]^{-1}
& =
{\partial \beta \over \partial z} \bigg|_{z =\phi(x,y)}\\
& = \ell Q + {\partial \phil \over \partial z}\bigg|_{z =\phi(x,y)} \\
& = \ell Q + {\partial b \over \partial x} \left({\partial \calh_\ell \over \partial x}\right)^{-1}.
\end{aligned}
$$
Since $Q$ is a Hurwitz matrix, there exists a matrix $M \succ 0$ such that $MQ+Q^\top M = - I$. By selecting a sufficiently large $\ell>0$, and noting $\calh_\ell$ independent of the parameter $\ell$, we guarantee {\bf A2} semi-globally.
\end{proof}

\begin{remark}\label{rem:sdo}\rm
The above result can be extended to more general cases---multi-input and nonautonomous. Furthermore, $n_\ell = n_\chi$ is the most widely studied case in high-gain observers. If a given system is strongly differentially observable with $n_\ell > n_\chi$, we need to add an augmentation dynamics as done in Section \ref{sec:3d}.
\end{remark}

\subsection{Connections with the full-order observer in \cite{MAN}}


Our previous paper \cite{MAN} uses a similar underlying mechanism to design full-order contracting observer for the system
\begequ
\label{NLS2}
 \dot \bfx = f(\bfx,u),\;  y = h(\bfx) \quad \bfx\in \rea^n, \; y\in \rea^{n_y}.
\endequ
To this end, we aim to find a coordinate change $z=\phi(\bfx)$ with $\dot z= f_z(\bfx,y,u)$ with some function $f_z$, which may be non-unique, but satisfying 
\begin{itemize}
    \item[\bf C1$'$] ({\em correctness}) $f_z(\bfx,h(\bfx),u)= \Phi(\bfx) f(\bfx,u)$, with $\Phi:= {\partial \phi \over \partial \bfx}$ full rank;
    \item[\bf C2$'$] ({\em contraction}) $\Phi^\top P^{-1} F + F^\top P^{-1} \Phi\prec 0$, with constant $P\succ 0$ and $F:= {\partial \over \partial \bfx} f(\bfx, y,u)$ and viewing $(u,y)$ as exogenous signals. 
\end{itemize}

Then, a full-order observer in the \emph{original coordinate} is given by
\begequ
\label{FOO}
\dot{\hat \bfx} = \Phi(\hat \bfx)^{-1} f(\hat\bfx,y,u).
\endequ
Its extension to discrete time and sampled-data versions were also considered in \cite{MAN}.

In Theorem \ref{thm:reduced} and Proposition \ref{prop:convex-observer-immersion} we extend the results in \cite{MAN} to the reduced-order observer design in transformed coordinates with arbitrary dimension $n_\xi \ge n_x$, thus enlarging the domain of applicability to more nonlinear systems. Here are some key differences between them.

\begin{itemize}
    \item[-] For the full-order case in \cite{MAN}, the system dynamics in the transformed $z$-coordinate is partially or virtually contracting \cite{WANSLO}, since the output $y$ is explicitly a function of the state $z$. But it is not the case for the reduced-order observer design in this paper.
    
    \item[-] The convex condition {\bf H3}, unlike the case in \cite{MAN}, introduces a nonlinear output injection term $\varphi(y)$, making it quite general. Indeed, coordinate changes in the form of $\phi(x) = Px + \varphi(y)$ are widely adopted in the field of nonlinear observers \cite{ASTbook}. 
    
    \item[-] An advantage of the full-order design \cite{MAN} is its absence of computing the left inverse mapping, and there is no need to guarantee the ``integrability'' of $\Phi(\bfx)$.
\end{itemize}

\subsection{Connections with I\&I observers}

To the best of our knowledge, the most comprehensive design approach of nonlinear reduced-order observers until now may refer to the I\&I observers, proposed in \cite{KARetal} and later elaborated in \cite[Chapter 5]{ASTbook}, by means of rendering attractive an appropriately selected invariant manifold in extended state space. The key extra degree of freedom given in I\&I observers relies on a coordinate transformation in order to be able to get a (generally nonlinear) asymptotically stable error dynamics. The procedure is similar to the one done in Section \ref{sec:3a}. In this subsection, we will clarify the connections and differences between the proposed designs and I\&I observers. Let us recall the main result in I\&I observers.\footnote{In \cite{KARetal} all the transformations, {\em i.e.} $\rho$ and $\varphi$, are allowed to be time-varying functions. We here adopt the time-invariant case to simplify the presentation.}

\begin{proposition}
\label{prop:i&i}\rm \cite{KARetal}
Consider the system \eqref{eq:syst} with forward complete solutions. Suppose that there exist mappings 
$
\rho: \rea^{n_\xi}\times \rea^{n_y} \to \rea^{n_\xi}, \quad
\varphi: \rea^{n_x} \times \rea^{n_y} \to \rea^{n_\xi}
$
with a left inverse $\varphi^{\tt L}(\cdot,y)$ such that the following holds.
\begin{itemize}
\item[\bf I1] For all $(y,\xi)\in \rea^{n_y} \times \rea^{n_\xi}$, 
$
\det(\nabla_\xi \rho) \neq 0.
$
\item[\bf I2] The system
\begin{equation}
\label{error-dyn}
\begin{aligned}
\dot e
=&
\left.\left[{\partial\varphi \over \partial y}(x,y)f_y(x,y,u)
\blue{+}
{\partial \varphi \over \partial x}(x,y)f_x(x,y,u)\right]\right|_{x=x}^{x=\hat x}
\\
& +
{\partial\rho \over \partial y}[f_y(x,y,u) - f_y(\hat x,y,u)]
\end{aligned}
\end{equation}
is asymptotically stable at the origin uniformly in $x,y$ and $u$ with $\hat x := \varphi^{\tt L}(\varphi(x,y)+e,y)$. Then, selecting
$$
\begin{aligned}
N(\cdot) & =  \Big({\partial \rho \over \partial \xi}\Big)^{-1}
\left[
{\partial \varphi \over \partial x}(\hat x,y) f_x(\hat x,y,u) \right.
\\
& \qquad +
{\partial \varphi \over \partial y}(\hat x,y) f_y(\hat x,y,u)
\left.
-{ \partial \rho \over \partial y}(\xi,y) f_y(\hat x,y,u)
\right]
\\
H(\cdot) & = \varphi^{\tt L}(\rho(\xi,y),y).
\end{aligned}
$$
makes the system \eqref{observer:general} a convergent observer.
\hfill$\triangleleft$
\end{itemize}
\end{proposition}

 We are now ready to show the connections between I\&I reduced-order observers and the proposed contracting observer design in Section \ref{sec:3}. In order to simplify the analysis, instead of {\bf I1} we additionally require a slightly stronger condition 
\begequ
\label{rho0}
 \big|\det(\nabla_\xi \rho) \big| > \rho_0, \quad \rho_0 >0.
\endequ
Then we have a global inverse $\rho^{\tt L}: \rea^{n_\xi} \times  \rea^{n_y} \to \rea^{n_\xi}$ such that $\rhol(\rho(\eta,y),y)=\eta$ for any $\eta \in \rea^{n_\xi}$.

We have the following.

\begin{proposition}
\label{prop:connection-i&i}\rm
Consider the system \eqref{eq:syst} admitting an I\&I observer identified in Proposition \ref{prop:i&i}, assuming that the mappings satisfy \eqref{rho0} and {\bf I2} with $n_\xi = n_x$. Then, a variant of {\bf A2}, namely
\begin{itemize}
  \item[\bf A2$'$] The LTV system
$
  {d\over dt} \delta \xi = {\partial f_z \over \partial x}(x,y,u) \left[ {\partial \phi \over \partial x}(x,y) \right]^{-1} \delta \xi
$
is asymptotically stable uniformly along any possible trajectory of $(x,y,u)$.
\end{itemize}
holds, and the given system has a contracting reduced-order observer in the sense of Theorem \ref{thm:reduced}. Furthermore, the obtained observer \emph{exactly} coincides with the I\&I observer by selecting
$\phi(x,y) = \rhol(\varphi(x,y),y).$
\end{proposition}
\begin{proof}
For convenience, we define a new intermediate variable as $\eta := \varphi (x,y)$. Invoking \eqref{rho0} and the left invertibility of $\varphi(x,y)$, we define a mapping $\phil: \rea^{n_x}\times \rea^{n_y} \to \rea^{n_x}$ as
$$
\phil(z, y) := \varphi^{\tt L}\left( \rho( z,y  ), y \right),
$$
and with some substitutions we have
$$
\begin{aligned}
\phil(z,y)\big|_{z = \phi(x,y)} 
=  \varphi^{\tt L}(\rho(\phi(x,y),y),y)
= x,
\end{aligned}
$$
thus $\phil$ being the left inverse of the composite function $\phi(\cdot)$.

In {\bf I2}, the variable $e$ is, indeed, the estimation error in the $\eta$-coordinate, {\em i.e.},
$
e = \rho(\xi,y) - \varphi(x,y).
$
Since the mapping $\rho(\cdot)$ identifies a diffeomorphism, one may equivalently define the error in the $\xi$-coordinate as
$$
e_z := \xi - \rhol(\varphi(x,y),y),
$$
implying that {\bf I2} is equivalent to the uniform asymptotic stability of the dynamics of $e_z$ with respect to the origin.

The dynamics of $e_z$ is given by
\begequ
\label{ez}\small
\begin{aligned}
\dot{e}_z & = N(\xi,y,u)  - \left[{\partial \rhol \over \partial \eta } \right] \left[ {\partial \varphi \over \partial x} f_x(x,y,u) + {\partial \varphi \over \partial y} f_y(x,y,u) \right] \\
& \quad - {\partial \rhol \over \partial y}f_y(x,y,u)
\\
& =  N(\xi,y,u)  - \left[{\partial \rho \over \partial z } \right]^{-1} \left[ {\partial \varphi \over \partial x} f_x + {\partial \varphi \over \partial y} f_y\right] + \left[ {\partial \rho \over \partial z } \right]^{-1} {\partial \rho \over \partial y} f_y\\
& =\left(  \left({\partial \rho \over \partial z } (\phi(x,y),y)\right)^{-1} \times \right. \\
&   \quad  \left( {\partial \varphi \over \partial x} f_x(x,y,u) + {\partial \varphi \over \partial y} f_y(x,y,u) - {\partial \rho \over \partial y} f_y(x,y,u)\right)  \Bigg)
\Bigg|^{\hat x}_{x}
\\
& = \left({\partial \phi \over \partial x}(x,y)f_x(x,y,u) + {\partial \phi \over \partial y}(x,y) f_y(x,y,u) \right)\bigg|^{\hat x}_{x}
\\
& = f_z(\hat x, y,u) -f _z(x,y,u)
\\
& =  f_z( \phil(\phi(x,y)+e_z, y), y,u) -f _z(x,y,u)
\\
& = \int_0^1 {\partial f_z \over \partial x}\left[ {\partial \phi \over \partial x} \right]^{-1} \big( \phi(x,y) + \mu e_z ,y,u\big) d\mu \cdot e_z
\end{aligned}
\endequ
where in the third equation we substitute the mapping $N(\cdot)$ in {\bf I2}, with $\hat x = \phil(\phi(x,y)+e_z, y)$. 

Now we consider a line segment $\gamma(s)$ between two initial conditions $\xi(0)$ and $z(0)= \phi(x(0),y(0))$ defined by the parameterization
$$
\gamma(s) = s \xi(0) + (1-s) z(0), \quad s\in [0,1].
$$
We denote $Z(t;\gamma(s),y,u)$ as the solution of the dynamical system
$$
\dot z = f_z(\phil(z,y),y,u)
$$
under the exogenous signals $(u,y)$ from the initial condition $\gamma(s)$ with $s\in [0,1]$. We define the infinitesimal displacement $\delta \xi$ as
$
\delta\xi = {{\partial Z(t;\gamma(s),y,u)} \over \partial s},
$
the dynamics of which is governed by the system in {\bf A2$'$}, with $x(t)=\phil(Z(t;\gamma(s),y,u),y)$. From the definition of $\delta \xi$, we have
$$
\xi(0) = z(0) + \delta \xi(0).
$$
Invoking \eqref{ez}, it then yields $\xi(t) \equiv z(t) + \delta\xi(t)$, and noting the uniform asymptotic stability of the $e_z$-system, thus
$$
|\delta \xi(t)| \le \kappa(|\delta\xi(0)|,t), \quad  t\ge 0
$$
for some function $\kappa$ of class $\mathcal{KL}$.
\end{proof}

Some further discussions are in order.

\begin{itemize}
  %
  \item[1)] The proposed designs---including Theorems \ref{thm:reduced}-\ref{thm:convex-observer}---provide ``more constructive'' solutions to the problem of reduced-order observer design, compared with the I\&I methodology. This is because the conditions in the proposed methods only constitute of some differential inequalities, but enjoy sufficient degrees of freedom. In particular, the results in Theorem \ref{thm:convex-observer} and Proposition \ref{prop:convex-observer-immersion} involve some computationally efficient convex conditions.
  \item[2)] In Theorem \ref{thm:reduced}, we need to search for two mappings---the coordinate transformation $\phi(x,y)$ and the metric $M$, and in Theorem \ref{thm:convex-observer} we need one more mapping $f_z$ to achieve convex relaxation. In contrast, when designing I\&I observers we need to search simultaneously for four mappings---the coordinate change $\varphi(x,y)$, its left inverse $\varphi^{\tt L}(z,y)$, another mapping $\rho$ and a Lyapunov function (to show convergence).
  \item[3)] If the mapping $\rho(\cdot)$ satisfies \eqref{rho0} instead of {\bf I1}, as clearly shown in the proof of Proposition \ref{prop:connection-i&i}, it is unnecessary to introduce the mapping $\rho$ in reduced-order observer design. We are able to ``gather'' all the degrees of freedom from $\rho$ and $\varphi$ in the transformation $\phi$.
\end{itemize}

\subsection{Connections with transverse contraction}

In this subsection, we study the link between the proposed designs and transverse contraction, a notion introduced in \cite{MANSLOscl} to analyze attractive limit cycles. It will be shown that designing a reduced-order observer is equivalent to find a coordinate transformation in which the system is transversely contracting.

We slightly extend \cite[Definition 1]{MANSLOscl} as follows.

\begin{definition}
\label{def:trans}\rm ({\em transverse contraction})
The forward complete system \eqref{chi} is said to be transversely contracting with rate $\lambda>0$ (or transversely asymptotically contracting) w.r.t. the function $\psi$, where $\psi: \rea^{n_\chi} \to \rea^{r}$ ($1\le r< n_\chi$), if for any initial condition pair $(\chi_a,\chi_b) \in \rea^{2n_\chi}$ and input $u$ we have
\begin{equation}
\label{trans-contraction}
 |\psi(\chi(t;\chi_a,u)) - \psi(\chi(t;\chi_b,u))| \le e^{-\lambda t} b(\chi_a,\chi_b)
\end{equation}
for some function $b(\cdot) \ge 0$ with $b(\chi,\chi)=0$ or
\begin{equation}
\label{trans-contraction2}
|\psi(\chi(t;\chi_a,u)) - \psi(\chi(t;\chi_b,u))| \le \kappa(|\chi_a - \chi_b|,t )
\end{equation}
for transverse asymptotic contraction, with $\kappa$ of class $\cal KL$.
\hfill$\triangleleft$
\end{definition}

It is clear that the definition in \cite{MANSLOscl}, tailored for orbital stability, is a particular case of Definition \ref{def:trans}. Considering the simple case that a given autonomous system has a unique attractive limit cycle represented by the function $q(\chi)=0$ with $q: \rea^{n_\chi} \to \rea^r$, we recover the above definition by choosing $\psi$ as $q(\chi)$. We provide a sufficient condition to achieve transverse asymptotic contraction below. The condition is motivated by the concept of \emph{asymptotic phase} in dynamical systems and also used in \cite[Corollary 2]{MANSLOtac} to verify the existence of asymptotic invariant manifold, also similar to what we obtain as follows.

\begin{proposition}
\label{prop:ccm}\rm
Consider the system \eqref{chi} forward invariant in a compact set $\cale:=\cale_x\times\cale_y$. Suppose there exists a mapping $\psi: \rea^{n_\chi} \to \rea^r$ ($1\le  r < n_\chi$) with $\nabla \psi$ full rank, and a uniformly bounded metric $P(\chi) \in \rea_{\succ0}^{n_\chi \times n_\chi}$ such that
\begin{equation}
\label{ccm}
{\partial \psi \over \partial \chi}
\left(
 \partial_f P + P(\chi) {\partial f\over \partial \chi} + {\partial f\over\partial \chi}^\top P(\chi)
\right)
{\partial \psi \over \partial \chi}^\top \prec 0
\end{equation}
holds. Then,
\begin{itemize}
\item[1)] the system \eqref{chi} is transversely asymptotically contracting with respect to $\psi$.
\item[2)] it admits a trivial functional observer
$
\dot{\hat \chi} = f(\hat\chi,u)
$
with respect to the function $\psi$.
\item[3)] if there exists a class $\mathcal{K}_\infty$ function $\rho_y$ (parameterised by $y$) satisfying the injectivity
$
|x_a-x_b| \le \rho_y\big( |\psi(x_a,y) - \psi(x_b, y)| \big), ~
\forall(x_a,x_b)\in B^2_\epsilon(\cale_x)
$
uniformly in $y$ for some $\epsilon>0$, the estimate
\begin{equation}\label{psil}
\hat x = \psi^{\tt L}(\psi(\hat \chi),y):= \underset{\eta \in B_\epsilon(\cale)}{\mbox{argmin}}\big| \psi(\eta,y) - \psi(\hat \chi) \big|
\end{equation}
guarantees \eqref{conv:general}.
\hfill$\triangleleft$
\end{itemize}
\end{proposition}

Its proof is given in Appendix A. We here make some comparisons between the results in Section \ref{sec:3} and transverse contraction.

\begin{itemize}
  \item[1)] The mapping \eqref{eq:T}
with $\phi$ defined in  Theorem \ref{thm:reduced}, is a diffeomorphism, and thus we may rewrite the full system dynamics in the $(z,y)$-coordinate as
$$
\Sigma_z: 
    \dot z  = f_z (z,y,u) , \quad \dot y  = f_y(z,y,u).
$$
If the system \eqref{chi} has a convergent reduced-order observer in the sense of Theorem \ref{thm:reduced}, the condition {\bf A2} makes the system $\Sigma_z$ transversely asymptotically contracting w.r.t. the linear mapping $\psi$ and the metric defined as
$$
\psi= \begmat{~I_{n_x} ~\Big| ~0_{n_y \times n_y}~} \begmat{z \\ y},
~
P = \begmat{P_z(z) & P_{yz}(z,y) \\ P_{yz}^\top (z,y) & P_{y}(z,y)}
$$
where $P_z(z)$ defined in  Theorem \ref{thm:reduced} (with a slight abuse of notations), and $P_{yz}$ and $P_{yy}$ are any mappings which guarantee $P$ a uniformly bounded metric.
%
%
\item[2)] The system \eqref{chi} in its original coordinate does not verify transverse contraction. The main difference between Theorem \ref{thm:reduced} and Proposition \ref{prop:ccm} is the way how we use $y$ and its estimate $\hat y$. In the latter, we \emph{do not} try to make full use of the information contained in the measured output when designing the functional observer $\dot{\hat \chi} = f(\hat\chi,u)$, making the sufficient condition relatively conservative. Meanwhile, transverse contraction implies the existence of a functional observer.
\item[3)] Some definitions similar to transverse contraction have been studied in the literature. In \cite{FORSEP}, the authors introduce horizontal contraction, a notion defined on \emph{tangent space}, to study the incremental stability along specific distributions. Here, the transverse contraction is defined on state space. In \cite{JAFetal} the concept of semi-contraction is recently proposed, which can be regarded as a particular case. In Appendix B, we give an alternative sufficient condition to transverse contraction.
\end{itemize}

\section{Examples}
\label{sec5}

In this section, we use the proposed methods to design reduced-order observers for four examples, including a numerical one and three physical benchmarks. 

\subsection{A polynomial example}

The first example is a polynomial system studied in \cite{FANARC,MAN}, with full-order observers presented therein. We use this example to show how to apply Theorem \ref{thm:convex-observer} to design an observer via convex optimization. The dynamics is given by
\begin{equation}\label{example:1}
\begin{aligned}
\begmat{\dot x_1 \\ \dot x_2} & = \begmat{x_1 - {1\over3}x_1^3 - x_1x_2^2 \\ x_1 - x_2 -{1\over3}x_2^3 - x_2x_1^2}, \quad
\dot y  = x_1.
\end{aligned}
\end{equation}
We aim to design a reduced-order observer using {\bf H3} in Theorem \ref{thm:convex-observer}, with mappings $P \in \rea^{2\times 2}_{\succ 0}$, $f_z$ and $\varphi: \rea \to \rea^2$ to be determined. We search polynomials using a positivstellensatz \cite{PAR} as 
$$
\begin{aligned}
f_{z1}(x,y), f_{z2}(x,y) \in \rea_{3}[x,y],
~
\varphi_1(y),\varphi_2(y) \in \rea_{2}[y]
\end{aligned}
$$
and select $\lambda =1$. A reduced-order observer is found, via Yalmip \cite{LOF} and Mosek, in 0.61 seconds on a desktop with 3.00 GHz Intel Core i7-9700 CPU and 3 GB RAM. The obtained mappings are $P=\diag(0.6370,0.6369)$ and $\varphi(y) = \col(-2.1872y, -0.6368y)$, with the observer given by
$$
\begin{aligned}
  \dot \xi  =  \begmat{0.6370(\hat x_1 - {1\over3}\hat x_1^3 - \hat x_1\hat x_2^2) \\0.6369(\hat x_1 - \hat x_2 -{1\over3}\hat x_2^3 - \hat x_2 \hat x_1^2)} ,~
  \hat x  = P^{-1}(\xi - \varphi(y)).
\end{aligned}
$$
We show in Fig. \ref{fig:1} the simulations with initial conditions $x(0) = [3\quad 5]^\top$, $y(0)=-4$ and $\xi(0) =  [0 \quad 0]^\top$, where high-frequency ``measurement noise'' has been added in the output channel. It is generated by
the block ``Uniform Random Number'' in Matlab/Simulink, with sample time 0.001 and the interval $[-0.02,0.02]$. As we can see, the estimates converge to their true values after a short transient stage, and the observer is not sensitive to measurement noise.
\begin{figure}
    \centering
    \includegraphics[width=\linewidth]{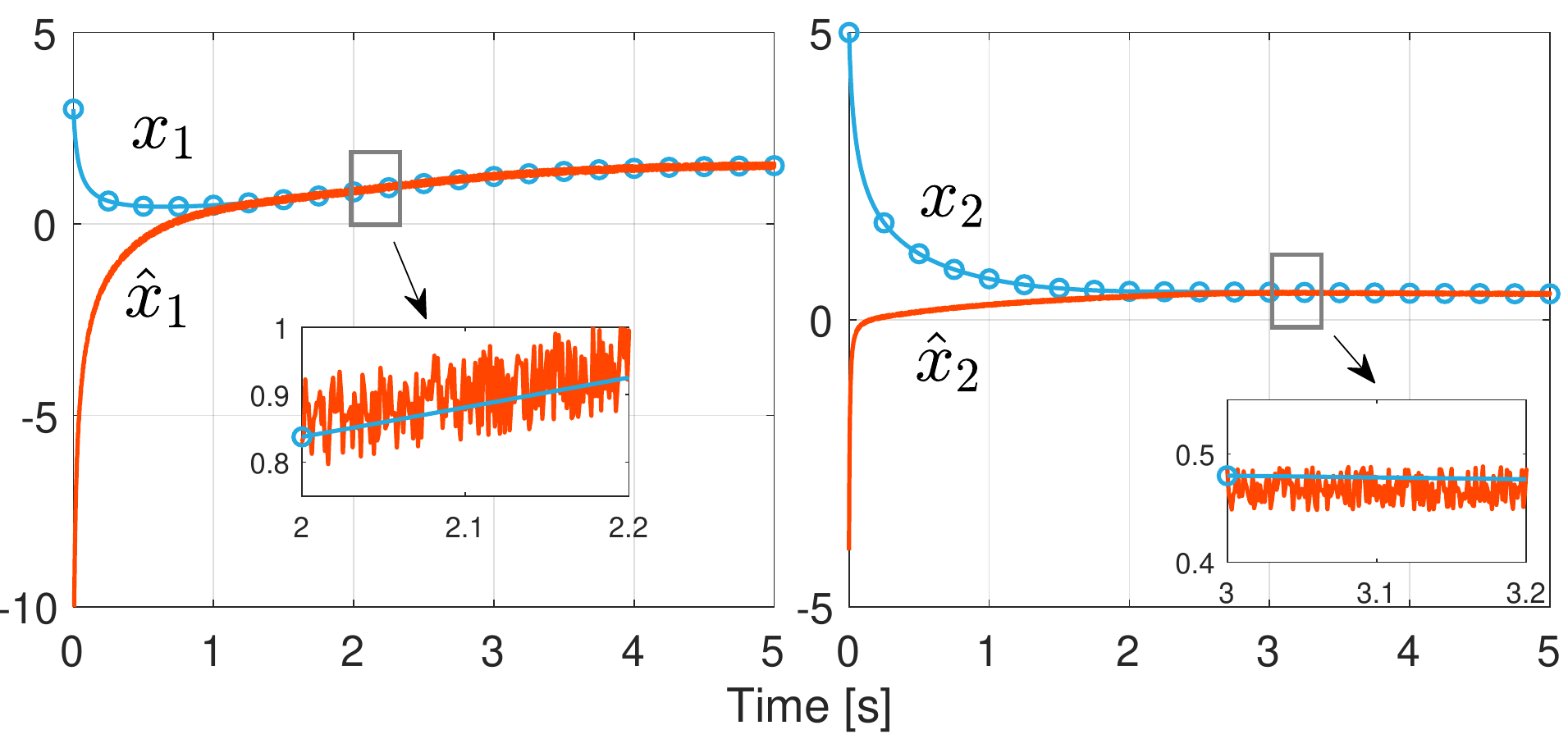}
    \caption{Simulation results of the polynomial system \eqref{example:1}}
    \label{fig:1}
\end{figure}

\subsection{Magnetic levitation system}

In the second example, we study an electromechanical system---magnetic levitation (MagLev) model, which is given by
\begin{equation}
\label{maglev}
\begin{aligned}
\begmat{\dot\lambda \\ \dot q \\ \dot p}
= 
\begmat{ {R \over k} (q-c) \lambda + u \\ {1\over m} p \\{1\over 2k} \lambda^2 - mg }
\end{aligned}
\end{equation}
with $\lambda $ the flux linkage, $q \in (-\infty,c)$ the position, considering physical constraints, $p$ the momenta and $u$ the input voltage, where $R,m,c,k>0$ are some physical parameters \cite{YIetalejc}. Here, we assume only the position $y=q$ measurable, with unknown states $x= \col(\lambda, p)$.\footnote{See \cite{YIetalejc} for a more challenging task of sensorless estimation, assuming only current $i:= {1\over k}(c-q)\lambda$ measured.}

The MagLev model is indeed a polynomial system, and thus the convex formulations in Theorem \ref{thm:convex-observer} are applicable. However, we provide an alternative of \emph{recursive} searching design. In the first step, we consider a two-dimensional system with $\lambda$ as the unknown state, and the measured state $y$. Then using {\bf H3} to search for parameter $P_1>0$ and mapping $\varphi_1 : \rea \to \rea$ under the constraint $q<c$. It is trivial to get an admissible solution $P_1=1$ and $\varphi_1=0$, and we are able to design an observer to get exponentially convergent estimate of $\lambda$. In the second step, we regard $p$ as the unmeasured state, $y$ the output and $\lambda$ a ``known'' input, and then solve the convex optimization in Theorem \ref{thm:convex-observer} again. We summarize the result below.

\begin{proposition}
\label{prop:maglev}\rm
Consider the MagLev model \eqref{maglev} with $q<c$. The observer
\begin{equation}
\label{observer:maglev}
  \dot\xi = \begmat{{R\over k}(y-c) \xi_1 + u\\ -{\ell \over m}\xi_2 -{\ell^2\over m}y - mg +{1\over 2k}\xi_1^2},
  \quad
  \hat x = \xi + \begmat{0 \\ \ell y}
\end{equation}
with any gain $\ell>0$, guarantees \eqref{conv:general} globally. \qed
\end{proposition}

In \cite{YAGYAZ} the authors present an observer, via contraction analysis, with only $y=q$ measured, as done here. However, that design only guarantees local convergence. We give the simulation results of both the proposed observer \eqref{observer:maglev} and the full-order observer in \cite{YAGYAZ} with all physical parameters borrowed from \cite{YIetalejc}, and $\ell =0.5$ for \eqref{observer:maglev} and $r_0=5,~k_0=600$ for the design in \cite{YAGYAZ}. We consider the system initial condition $x(0)=[0.003,0,0]^\top$, and two sets of observer initial conditions. In the first, we assume having a good initial guess, {\em i.e.}, $\hat \xi(0)=[0, 0]^\top$ for the proposed observer and $\hat x(0)=[0.0035, 0, 0]^\top$ for the other. As we can see, both of them get satisfactory performance. In the second, we consider large initial condition differences, {\em i.e.}, $\xi(0)=[0.010, 0.010]^\top$ and $\hat x(0)=[0.006,0.002, 0.0078]^\top$. For this case, our proposed observer states still converge to the true ones, but the observer in \cite{YAGYAZ} diverges, showing its locality.

\begin{figure*}
    \centering
    \includegraphics[width=0.8\linewidth]{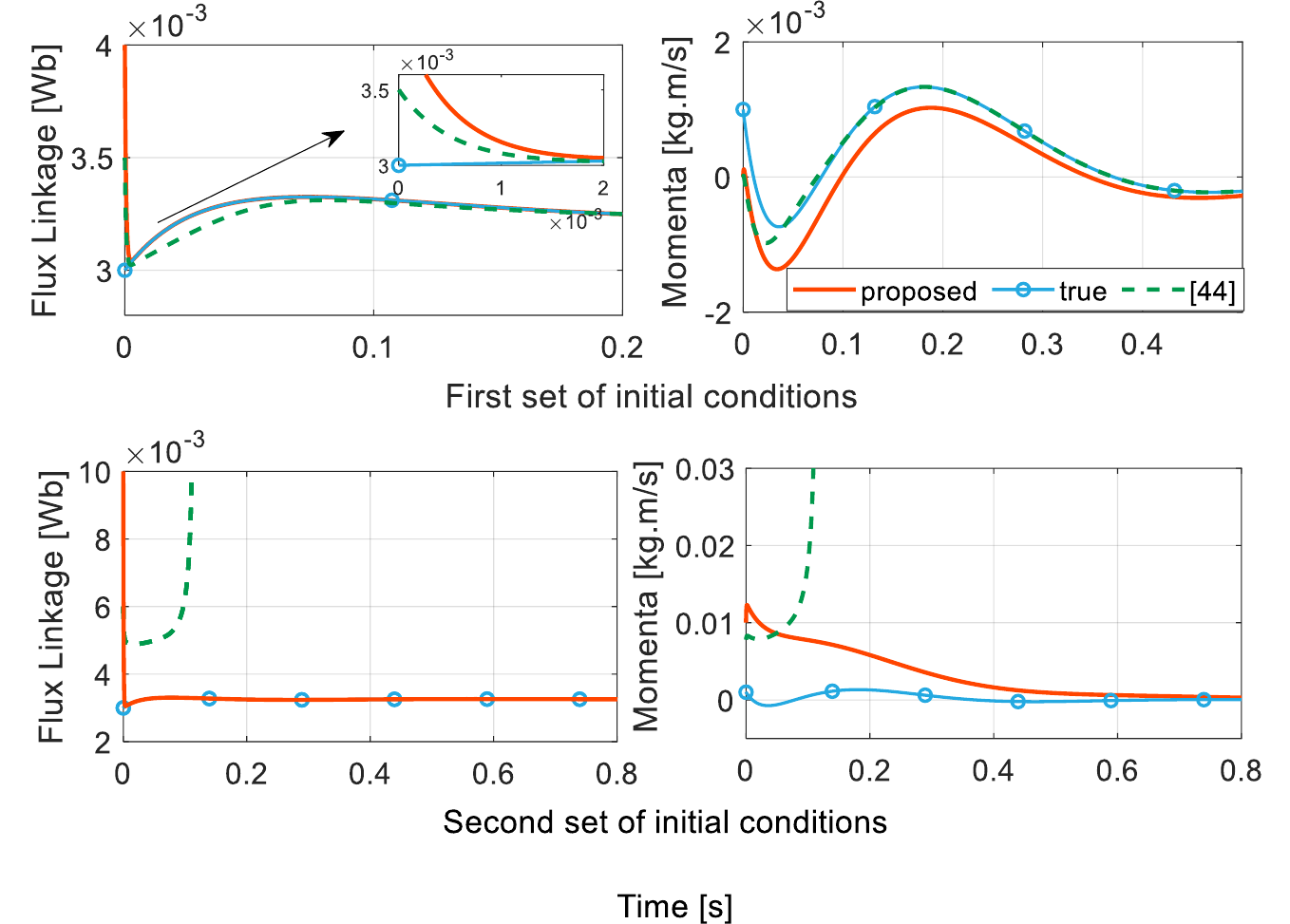}
    \caption{The unmeasured states and their estimates from the proposed observer \eqref{observer:maglev} and the structure preserving observer in \cite{YAGYAZ}, which is a local design, for the MagLev model \eqref{maglev}}
    \label{fig:2}
\end{figure*}

\subsection{Inverted pendulum on cart system}

In this subsection, we consider a mechanical model of cart-pendulum system, the model of which is given by
%
%
\begin{equation}
\label{pendulum}
\begin{aligned}
 \dot q  = \Psi (q)p ,\quad
 \dot p  = \Psi^\top(q) [ \nabla V(q) - G(q)u ]
\end{aligned}
\end{equation}
with
$$
\Psi(q) =
\begmat{
{\sqrt{m} \over \sqrt{m -b^2 \cos^2(q_1)} } & 0\\
{-b \cos(q_1) \over \sqrt{m} \sqrt{m-b^2 \cos^2(q_1)} } &  {1\over \sqrt m}
}
, \;
G =\begmat{0 \\ 1}
$$
and $V(q) = a\cos(q_1)$, where $q \in \mathbb{S}\times \rea$ is the generalized position---the pendulum angle $q_1$ and the horizontal position of the car, $p\in \rea^2$ is a momenta-like variable with $\Psi(q)p$ the generalized velocities, with parameters $a,b,m>0$. We assume that only $y=q$ is measured, and the task is to estimate the internal state $x = p$, equivalently the speed. It is well known that high-gain observer is applicable to this task with a \emph{semi-global} domain of attraction by increasing the adaptation gain. We here use {\bf H3} to obtain a globally convergent observer \emph{without} high-gain injection.

In the design, we fix $P=I_2$ and get the transformation 
$
\phi(x,y) = x + \varphi(y),
$
with $\varphi(y)$ to be found. Note that $\Psi(q)$ is full rank, and the dynamics of $y$ is affine in $x$. The condition {\bf H3} motivates us to select $\varphi(y)$ satisfying the PDE
\begequ
\label{pde}
{\partial \varphi \over \partial y} (y)\Psi(y) = -\lambda I_2
\endequ
with $\lambda >0$. Using the symbolic solver, we get a feasible solution
\begin{equation}
\label{varphi:pendulum}
\varphi(y) =  -\lambda
\begmat{
   \int_{0}^{y_1} \sqrt{1- {b^2\over m}\cos^2(s) }ds
  \\
  {b\over \sqrt m} \sin(y_1)  + \sqrt m y_2
}.
\end{equation}
We are ready to present the observer design as follows.

\begin{proposition}
\label{prop:cart}\rm
Consider the system \eqref{pendulum} with $q$ measured. The reduced-order observer
\begin{equation}
\label{obs:pendulum}
\begin{aligned}
   \dot \xi & = - \lambda \xi + \lambda \varphi(y) + \Psi^\top(y)[\nabla V(y) - Gu]
   \\
   \hat x & = \xi - \varphi(y).
\end{aligned}
\end{equation}
guarantees 
   $
   \lim_{t\to\infty} |  \hat x(t) - p (t)| =0 
   $
  exponentially.
\end{proposition}
\begin{proof}
Since the function \eqref{varphi:pendulum} is the solution of \eqref{pde}, we thus have
$
F^\top + F + 2\lambda I_2 = 0
$
satisfying the condition {\bf H3}. Invoking
$$
\begin{aligned}
{d\over dt} (\phi(x,y)) & = {\partial \varphi \over\partial y}(y) \Psi(y) x + \Psi^\top(y)[\nabla V(y) + Gu]\\
& = -\lambda \phi(x,y)+ \lambda\varphi(y) + \Psi^\top(y)[\nabla V(y) + Gu],
\end{aligned}
$$
thus we get the observer above, such that $\hat x$ is an exponentially convergent estimate of $p = \Psi^{-1}\dot q$. 
\end{proof}

We also give the simulation results in Fig. \ref{fig:3} with parameters $m=1,a=1,b=0.1$, initial conditions $\xi(0)= [0,0]^\top$, $p(0)=[0.4,0.3]^\top$ and $q(0)=[\hal \pi -0.1, -0.1]$, and $u(t)=0.2\cos(t)$, using different adaptation gains. Clearly, a large parameter $\lambda$ implies a fast convergence speed.

\begin{figure}
    \centering
    \includegraphics[width=\linewidth]{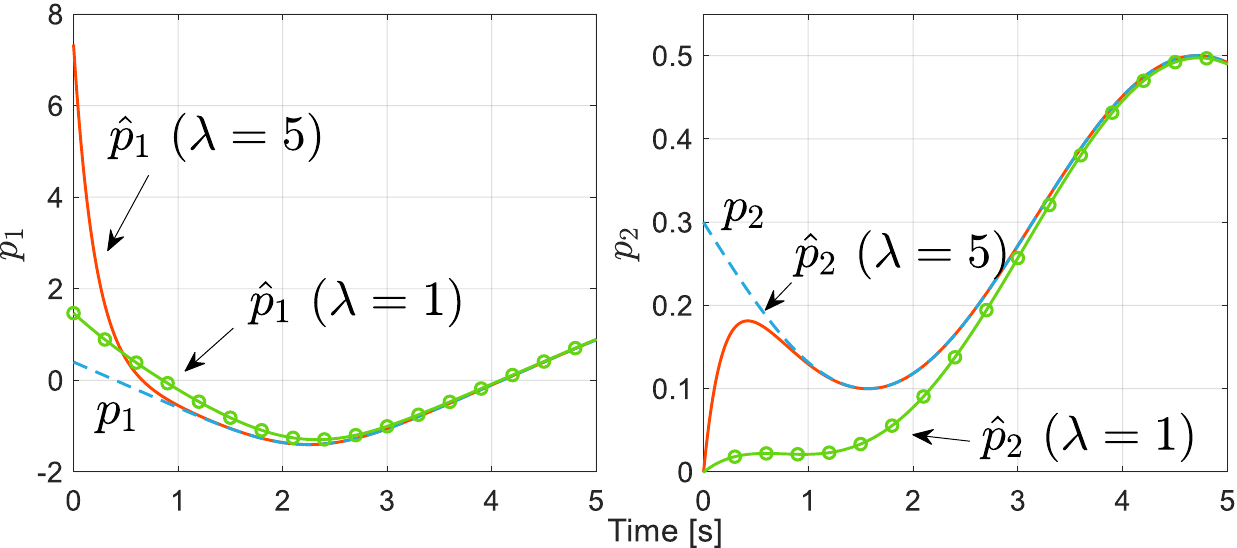}
    \caption{Estimation performance of the cart-pendulum system}
    \label{fig:3}
\end{figure}

\subsection{A biological reactor}

The last example is used to illustrate the case $n_\xi > n_x$ in Subsection \ref{sec:3d}. We now consider a model of reaction in \cite{RAPMAL}
\begin{equation}
\label{reactor}
\dot x = - {1\over k}\mu(x) y , \quad \dot y = \mu(x) y,
\end{equation}
with the states $(x,y)\in \rea^2_{>0}$, a parameter $k>0$ and $\mu(\cdot)$ a non-negative smooth function such that $\mu(0)=0$, and we refer the reader to \cite{RAPMAL} for the physical meanings of states and parameters. We here adopt the non-monotonic growth function 
$$
\mu(x) = r x \left(1-{x\over c}\right),~~x \in [0,c]
$$
with $r,c>0$ to characterize some inhibition effect of the reactor. To simplify the presentation, we assume all the positive constants equal to one. It was shown in \cite{RAPMAL} that $[y, f_y]$ is not injective as soon as $\mu(x)$ is non-monotonic, thus with the need to increase the observer dimension. Here, we follow Proposition \ref{prop:convex-observer-immersion} to provide a new design as follows.
\begin{proposition}
Consider the system \eqref{reactor}. The observer
$$
\begin{aligned}
\dot \xi  = - \xi + 
\begmat{\ln y \\ 2y+1 \\ y+y^2 \\
\hal \xi_2 - \hal(\xi_2^2 - 4(\xi_3-\xi_1+\ln y))^{1\over 2}
},~
\hat x = \xi_4 -y
\end{aligned}
$$
guarantees $\lim_{t\to\infty}|\hat x(t) - x(t)|=0$.
\end{proposition}
\begin{proof}
We design an augmentation $\dot w = Aw + f_a(y)$ with $A=-I_{3\times 3}$, the initial condition constraints $w_2(0)=w_3(0)=0,~w_1(0)=(x(0)+y(0))^2 + \ln (y(0))$. It should be underscored here that such initial conditions are \emph{not} used in the observer implementation. It is straightforward to verify for all $t\ge 0$
$$
\begin{aligned}
2z_4(t) = w_2(t) - [(w_2^2(t) - 4(w_3(t)-w_1(t)+\ln(y(t))]^{1\over 2},
\end{aligned}
$$
with $z_4:=x+y$. By selecting the transformation as
$
\phi(x,w,y) = \col(w, x+y) \in \rea^4,
$
it is straightforward to verify {\bf H1$'$}, {\bf H2$'$} and {\bf H3$'$}. 
\end{proof}

Our design is compared with the high-gain observer (HGO) proposed in \cite{RAPMAL}, which is given by
$$
\footnotesize\begin{aligned}
\dot \xi & = \begmat{\xi_2 \\ -{\xi}_3\min(e^{\xi_1}, \bar \xi_1)  \\
\underset{m,M}{\tt sat}\left(e^{\xi_1} (2\xi_2^2-\xi_3\tilde\phi(\xi))\right)}
 +
\begmat{-3\ell \\ -3\ell^2 \\3\ell\xi_3 \mathbb{1}_{\xi\le \ln \bar\xi_1} + \ell^3 \max(e^{-\xi_1}, {1\over \bar\xi_1})} e_y
\\
\hat x & = \begmat{e^{\xi_1} \\ \hal(1-\tilde\phi(\chi)) }, \quad
e_y := \xi_1 - \ln y
\end{aligned} 
$$
with some parameters $\bar\xi_1, m,M, \xi_2^\star$, $P(\xi_2):= \xi_2((\xi_2/\xi_2^\star)^2 - 3\xi_2/\xi_2^\star +1)$ and a function
$$\footnotesize
\tilde \phi(\xi) = 
\left\{
\begin{aligned}
&- 1 & & \mbox{if~}\xi_2> \xi_2^\star \mbox{~and~}\xi_3<-\xi_2 \\
&\left({\xi_2/\xi_2^\star}\right)^{2} - 3{\xi_2/\xi_2^\star} +1 &&
\mbox{if~}\xi_2 \in (0,\xi_2^\star] \mbox{~and~} \xi_3 < P(\xi_2) 
\\
&{\xi_3 / \xi_2} && \mbox{otherwise}.
\end{aligned}
\right.
$$
Clearly, the HGO is much more complicated, stemmed from Lipschitsian extensions. These two observers are compared via simulations, with the initial conditions $[0, 0, 0, 0.5]^\top$ in the proposed contracting observer, and $[1, 0.1, 0]^\top$, together with the gain $\ell =3$ for the HGO. Some simulations were done for the cases with and without measurement noise, shown in Fig. \ref{fig:reactor}. It is well known that the convergence  speed of HGOs mainly depends on the adaptation gain $\ell>0$. In order to make a fair comparison to evaluate their effects to noise, we tuned them, via changing $\ell$, such that their convergence times are similar. As expected, our new design overperforms the HGO, in particular in the presence of measurement noise, since the underlying mechanism of HGOs makes them highly sensitive to measurement noise. We underline that the effect of noise in the HGO becomes increasingly serious, since the studied trajectory verifies $x_2(t) \to 0$ and $\xi_2$ is its estimate, with $\xi_2$ appearing in the denominator of $\bar\phi(\xi)$. Besides, peaking phenomenon is clearly observed in the HGO, which usually requires some extra modification, {\em e.g.}, adding saturation function or projector \cite{KHAbook}, to avoid deleterious effects.

\begin{figure*}
    \centering
    \includegraphics[width=0.85\linewidth]{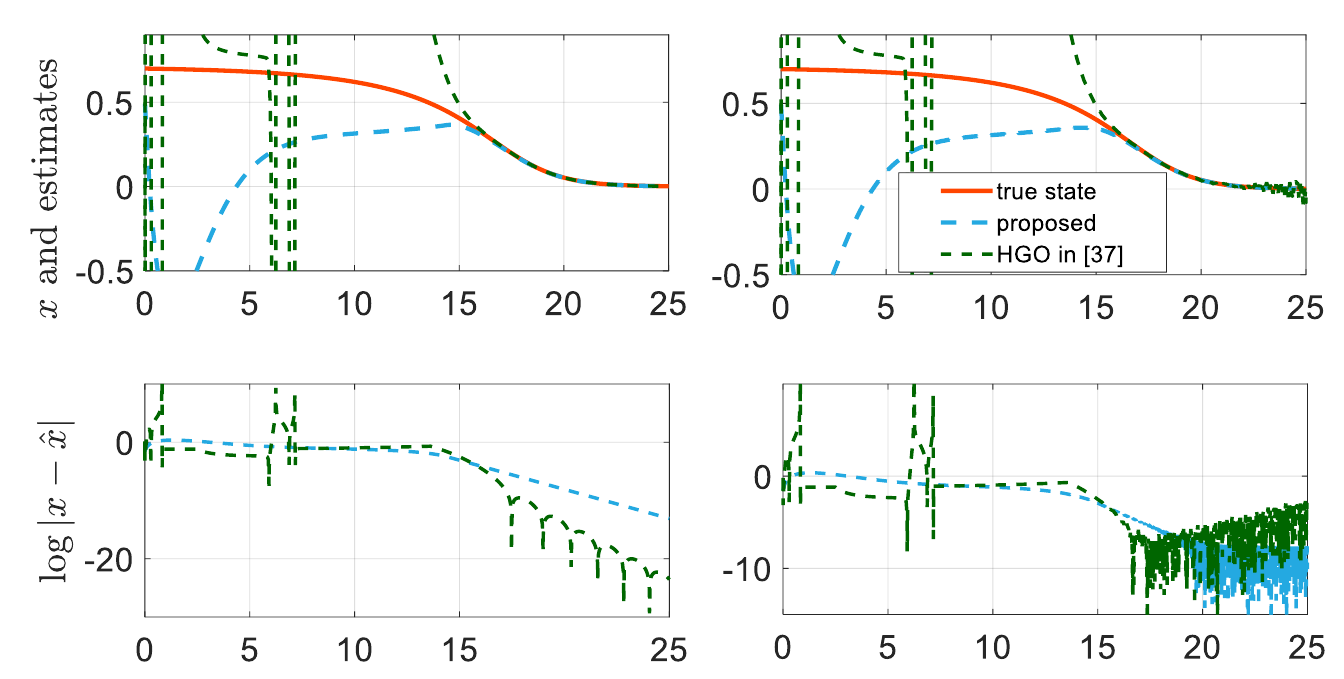}
    \caption{Simulation results of the reactor model with comparison to the high-gain observer in \cite{RAPMAL}}
    \label{fig:reactor}
\end{figure*}

\section{Concluding Remarks}
\label{sec6}

In this paper we address the problem of reduced-order observer design for nonlinear control systems by means of contraction analysis and convex optimization. As we have shown, searching for a change of coordinate---after which the system is (transversely) contracting---plays a key role in the constructive procedure, as well as providing sufficient degrees of freedom to characterize ``detectability'' for nonlinear systems. Then we formulate the obtained sufficient condition into a convex optimization problem by introducing the image as a new decision variable. We apply the main results to linear systems and strongly differentially observable nonlinear systems, showing generality of the proposed framework. Links to I\&I observers and transverse contraction are also clarified in the paper. At the end, we apply the method to several academic and practical models as illustrations. 

Along this research line, the following problems are underway.

\begin{itemize}
    \item[-] If the correctness is infeasible in some sense, the proposed method is still able to provide estimation with ultimate bounded errors. It is also interesting to study interval observers using the introduced framework.
    
    \item[-] After getting a convex sets of contracting observers, it is of practical interest to select the best or mix them in the presence of unavoidable measurement noise.
    
    \item[-] For nonlinear nonautonomous systems, the necessary condition on the existence of an asymptotic observer is still open. If it exists, it is natural to study the relationship between the necessary condition and the condition {\bf A2}.
    
    \item[-] Some practical models are not uniformly observable, but may be asymptotically estimated when the system trajectories satisfy additional excitation conditions. It is interesting to generalize the proposed framework to this case.
    
    \item[-] In the recent work \cite{BERAND}, the exogenous input $u(t)$ is involved in a coordinate transformation. If we allow this degree of freedom in the proposed method, it will definitely extend the realm of applicability. At the moment, such a time-varying transformation in our framework only makes sense conceptually, but with difficulty to implement it.
    \end{itemize}

\appendix

%
\section*{A. The Proof of Proposition \ref{prop:ccm}}
Invoking Wazewski theorem \cite{BER} and the fact that Euclidean space is contractible, we obtain that the Jacobian completion of $\nabla \psi (\chi)$ is solvable, {\em i.e.}, there exists a mapping $\tilde g: \rea^{n_\chi} \to \rea^{n_\chi \times (n_\chi -r)}$ such that
$$
\det \begmat{{\partial \psi \over \partial \chi}^\top(\chi) & \tilde{g}(\chi) } \neq 0.
$$
We can use Gram-Schmidt process to orthonormalize the column vectors of $\begmat{\nabla \psi(\chi) & \tilde g(\chi)}$ with the standard inner product in Euclidean space, and the obtained last $(n_\chi -r)$ vectors are denoted as $g(\chi)$. It is clear that $g(\chi)$ is uniformly full rank, and each of its column $g_i(\chi)$ verifies
\begin{equation}
\label{gi-in-im}
g_i(\chi) \in \mbox{ker}\left( {\partial \psi \over \partial \chi}(\chi)\right),\quad \forall i \in \{1,\ldots, n_\chi -r\}.
\end{equation}

Now we select an arbitrary piecewise continuous $\calc^1$ path $\gamma: [0,1] \to \rea^{n_\chi}$ between any two point $\chi_a =\gamma(0) $ and $\chi_b = \gamma(1)$. Consider a family of auxiliary systems
\begin{equation}
\label{aux-syst}
\begin{aligned}
{d\over dt} \Gamma(s,t) = f(\Gamma(s,t),u(t)) + g(\Gamma(s,t)) v(s,t),
\end{aligned}
\end{equation}
with $\Gamma(s,0) =\gamma(s)$, $\Gamma(\cdot) \in \rea^{n_\chi}$ and the \emph{control input} $v(s,t)$, where we regard $u(t)$ as an exogenous time-varying signal. By calculating its differential dynamics, we have
\begin{equation}
\label{diff_aux}
\begin{aligned}
& {d\over dt} {d\over ds} \Gamma(s,t) = g(\Gamma(s,t)) {d\over ds} v(s,t) \\
& \hspace{1.2cm}+ \left({\partial f \over \partial \chi}(\Gamma(s,t),u(t)) + \sum_{i=1}^{n_\chi -r} {\partial g_i \over \partial \chi } v_i(s,t) \right)
 {d\over ds}\Gamma(s,t)
\end{aligned}
\end{equation}
According to \cite[Theorem 1]{MANSLOtac} and the full-rank property of $g(\cdot)$, there exists a differential controller ${d\over ds} v(s,t) = \delta v (\Gamma(s,t),{d\over ds}\Gamma(s,t),t) $ such that the origin of the differential dynamics \eqref{diff_aux} is asymptotically stable uniformly in $s \in[0,1]$ with a quadratic Lyapunov function. Then, the controller
$
v(s,t) = - \int_s^1 \delta v (\Gamma(\sigma,t),{d\over ds}\Gamma(\sigma,t),t ) d \sigma
$
makes the close loop incrementally asymptotically stable, thus
\begin{equation}
\label{converge_phichi}
\lim_{t\to\infty} | \Gamma(1,t) - \Gamma(0,t) | =0.
\end{equation}

On the other side, since $\Gamma(1,0) = \hat{\chi}(0)$ and $v(1,t) = 0$, we have
\begin{equation}
\label{Gamma1}
\Gamma(1,t) = \hat\chi(t).
\end{equation}
We also have
\begin{equation}
\label{Gamma0}
\begin{aligned}
 & {d\over dt} \Big( \psi(\chi(t)) - \psi(\Gamma(0,t) \Big) \\
& \hspace{1cm} = {\partial \psi\over \partial \chi}(\chi(t)) f(\chi(t),u) \\
& \hspace{1cm} \quad  -
 {\partial \psi\over \partial \chi}(\Gamma(0,t))  \big(f(\Gamma(0,t),u) + g(\Gamma(0,t) v(0,t) \Big)
 \\
& \hspace{1cm}  \overset{\eqref{gi-in-im}}{=} {\partial \psi\over \partial \chi}(\chi(t)) f(\chi(t),u) -
 {\partial \psi\over \partial \chi}(\Gamma(0,t))  f(\Gamma(0,t),u) ,
\end{aligned}
\end{equation}
together with $ \Gamma(0,0)= \chi(0)$ it implying
\begin{equation}
\label{equiv_phi}
  \psi(\chi(t)) = \psi(\Gamma(0,t)) , \qquad \forall t\ge 0.
\end{equation}
Combining \eqref{converge_phichi}, \eqref{Gamma1} and \eqref{equiv_phi}, we prove the first claim. The claim 2) is trivial to verify by fixing $\chi = \chi(t;\chi_a,u)$ and $\hat\chi = \chi(t;\chi_b,u)$, verifying the second claim.

Regarding the last item, invoking the forward invariance assumption, we have $\psi(\chi) \in \psi(\cale)$ for all $t\ge 0$. Recalling the convergence \eqref{trans-contraction2}, for any $\epsilon >0$ we can always find a moment $t_\epsilon$ such that
$
\psi(\hat \chi(t))  \in B_\epsilon(\cale), ~ \forall t\ge t_\epsilon.
$
In terms of injectivity, the mapping defined in \eqref{psil} guarantees
$$
\psi(\hat \chi) \in B_\epsilon(\cale)
$$
implying
$$
|\psi^{\tt L}(\psi(\hat \chi),y) - \psi^{\tt L}(\psi(\chi),y) | \le \rho_y (|\psi(\hat \chi) - \psi(\chi)|).
$$
It, then, yields
$$
|\hat x(t) - x(t)| \le \rho_y\big( |\psi(\hat \chi(t)) - \psi(\chi(t)) |\big), ~ t\ge t_\epsilon.
$$
Combining \eqref{trans-contraction2}, we complete the proof.
\QED

\section*{B. An Alternative Condition to Transverse Contraction}

In this section, we give an alternative sufficient condition to transverse contraction, by using a \emph{semi-definite} Riemannian metric $W(\chi)$. A similar problem arises in \cite{WANetal}.

\begin{proposition}
\label{prop:alternative}\rm
Consider the system \eqref{chi} forward invariant in $\cale$ under input $u(t)$. Suppose there exists a positive semi-definite $W: \rea^{n_\chi} \to \rea^{n_\chi \times n_\chi}_{\succeq 0}$ with rank $r>0$, parameterized as $W(\chi) := \Psi(\chi)P(\chi)\Psi^\top(\chi)$ with $\Psi(\chi) \in \rea^{n_\chi \times r}$ and $P(\chi) \in \rea^{n_\chi \times n_\chi}$ both full rank.\footnote{A positive semi-definite matrix can always be represented in such a parameterization.} If the inequality
\begin{equation}\label{ineq:trans}
  \partial_f W(\chi) + {\partial f\over \partial \chi}^\top W(\chi) + W(\chi) {\partial f\over \partial \chi} \prec 0
\end{equation}
holds, and $\nabla \Psi_i(\chi) = (\nabla \Psi_i (\chi))^\top$ for $i =1 \ldots, r$, then the given system is transversely asymptotically contracting w.r.t. the function $\psi$ defined by
$
\psi (\chi) = \int_0^1 (\Psi(s\chi))^\top \chi ds.
$
\end{proposition}
\begin{proof}
We choose a differential Lyapunov function
$$
V(\chi,\delta\chi) = \delta \chi^\top W(\chi) \delta \chi,
$$
and define
$$
\delta \chi_h := \Psi(\chi) \Psi(\chi)^\dagger \delta \chi,
$$
and 
$$
\delta \chi_v := [I -\Psi(\chi) \Psi(\chi)^\dagger ] \delta \chi.
$$
For each $(\chi, \delta\chi) \in  \rea^{n_\chi} \times T\rea^{n_\chi}$, we have
$$
\dlc = \dlc_h + \dlc_v,
$$
and then $V(\chi,\delta\chi) = \dlc_h ^\top W(\chi) \dlc_h$,
and thus
$
V(\chi,\dlc) = V(\chi,\dlc_h)
$
The associated differential dynamics is
$
{d\over dt} \delta \chi
=
{\partial f(\chi,u)\over \partial \chi}  \delta \chi,
$
and for any nonvanishing $ \dlc$,
$$
\dot V(\chi,\dlc)
 = \dlc^\top
( \partial_f W + {\partial f\over \partial \chi}W + W{\partial f \over \partial \chi}^\top
) \dlc < 0,
$$
thus
\begequ
\label{eq:convergence_v}
\lim_{t\to\infty} V(\chi(t),\dlc(t)) =0.
\endequ

It is obvious that $\nabla \psi(\chi) = \Psi^\top(\chi)$. Consider a differential curve $\gamma:[0,1] \to \rea^{n_\chi}$ with $\gamma(s) \in \cale$, we parameterize two different initial conditions as $\chi_a(0) = \gamma(0)$ and $\chi_b(0) = \gamma(1)$. Then, it yields from \eqref{eq:convergence_v} the following implications.
$$
\begin{aligned}
 &\Longrightarrow
\; \lim_{t\to\infty }\int_{0}^1 \left| {d\over ds}\chi(t; \gamma(s),u)^\top \Psi(\chi(t; \gamma(s),u)) \right|_{P(\cdot)} ds =0
\\
&\Longrightarrow
\;
\lim_{t\to\infty }{d\over ds} \psi(\chi(t; \gamma(s),u))  =0
\\
& \Longrightarrow
\;
\lim_{t\to\infty} |\psi(\chi(t;\chi_a,u)) - \psi(\chi(t;\chi_b,u)) | =0
\end{aligned}
$$
where in the third implication we have used $\det(P(x))\neq 0$ and the boundedness of $P(\cdot)$. It completes the proof.
\end{proof}

\section*{Acknowledgement}
The authors would like to express their gratitude to three reviewers for their careful reading of our manuscript, and many thoughtful comments that helped improve its clarity.


\begin{thebibliography}{00}

\bibitem{AGHROU}
N. Aghannan and P. Rouchon, An intrinsic observer for a class of Lagrangian systems, \TAC, vol. 48, pp. 936--945, 2003.

\bibitem{ANDPRA}
V. Andrieu and L. Praly, On the existence of a Kazantzis-Kravaris/Luenberger observer, \SIAM, vol. 45, pp. 432--456, 2006.

\bibitem{ANDetal}
V. Andrieu, G. Besan\c{c}on and U. Serres, Observability necessary conditions for the existence of observers, \CDC, Florence, Italy, pp. 4442--4447, 2013.


\bibitem{ASTbook}
A. Astolfi, D. Karagiannis and R. Ortega, {\em Nonlinear and Adaptive Control with Applications}, Springer, London, 2007.

\bibitem{BER}
P. Bernard, {\em Observer Design for Nonlinear Systems}, Springer, Switzerland, 2019.

\bibitem{BERAND}
P. Bernard and V. Andrieu, Luenberger observers for nonautonomous nonlinear systems, \TAC, vol. 64, pp. 270--281, 2018.

\bibitem{BERetalsiam}
P. Bernard, V. Andrieu and L. Praly, Expressing an observer in preferred coordinates by transforming an injective immersion into a surjective diffeomorphism, \SIAM, vol. 56, pp. 2327--2352, 2018.

\bibitem{BES}
G. Besan\c{c}on (Ed.), {\em Nonlinear Observers and Applications}, Springer, Berlin, 2007.

\bibitem{BESHAM}
G. Besan\c{c}on and H. Hammouri, On uniform observation of nonuniformly observable systems, \SCL, vol. 29, pp. 9--19, 1996.



\bibitem{ENGKRE}
R. Engel and G. Kreisselmeier, A continuous-time observer which converges in finite time, \TAC, vol. 47, no. 7, pp. 1202--1204, 2002.

\bibitem{FANARC}
X. Fan and M. Arcak, Observer design for systems with multivariable monotone nonlinearities, \SCL, vol. 50, pp. 319--330, 2003.

\bibitem{FORSEP}
F. Forni and R. Sepulchre, A differential Lyapunov framework for contraction analysis, \TAC, vol. 59, pp. 614--628, 2014.

\bibitem{GAUKUPbook}
J.P. Gauthier and I. Kupka, {\em Determinstic Observation Theory and Applications}, Cambridge University Press, New York, 2001.

\bibitem{GOOMID}
G.C. Goodwin and R.H. Middleton, The class of all stable unbiased state estimators, \SCL, vol. 13, pp. 161--163, 1989.



\bibitem{JAFetal}
S. Jafarpour, P. Cisneros-Velarde and F. Bullo, Weak and semi-contraction theory with application to network systems, {\em ArXiv Preprint}, 2020. (\texttt{arXiv:2005.09774})

\bibitem{KARetal}
D. Karagiannis, D. Carnevale and A. Astolfi, Invariant manifold based reduced-order observer design for nonlinear systems, \TAC, vol. 53, pp. 2602--2614, 2008.

\bibitem{KAZKRA}
N. Kazantzis and C. Kravaris, Nonlinear observer design using Lyapunov's auxilary theorem, \SCL, vol. 34, no. 5, pp. 241--247, 1998.

\bibitem{KHAbook}
H.K. Khalil, {\em High-gain Observers in Nonlinear Feedback Control}, SIAM, Philadelphia, 2017.

\bibitem{KHAPRA}
H.K. Khalil and L. Praly, High-gain observer in nonlinear feedback control, \IJRNLC, vol. 24, no. 6, pp. 993--1015, 2014.

\bibitem{KREENG}
G. Kreisselmeier and R. Engel, Nonlinear observers for autonomous Lipschitz continuous systems, \TAC, vol. 48, pp. 451--464, 2003.

\bibitem{LEW}
D. Lewis, Metric properties of differential equations, {\em American Journal of Mathematics}, pp. 294--312, 1949.

\bibitem{LOHSLO}
W. Lohmiller and J.-J.E. Slotine, On contraction analysis for non-linear systems, \AUT, vol. 34, pp. 638--696, 1998.

\bibitem{LOF}
J. Lofberg, Yalmip: A toolbox for modeling and optimization in MATLAB, {\em Proc. of the CACSD Conf.}, Taipei, Taiwan, 2004.

\bibitem{LUE}
D.G. Luenberger, Observing the state of a linear system, {\em IEEE Trans. Mil. Electron.}, vol. 8, pp. 74--80, 1964.

\bibitem{LUEtac71}
D.G. Luenberger, An introduction to observers, \TAC, vol. 16, no. 6, pp. 596--602, 1971.

\bibitem{MAN}
I.R. Manchester, Contracting nonlinear observers: Convex optimization and learning from data, \ACC, Milwaukee, USA, pp. 1873--1880, June 27-29, 2018.

\bibitem{MANSLOscl}
I.R. Manchester and J.-J.E. Slotine, Transverse contraction criteria for existence, stability, and robustness of a limit cycle, \SCL, vol. 63, pp. 32--38, 2014.

\bibitem{MANSLOtac}
I.R. Manchester and J.-J.E. Slotine, Control contraction metrics: Convex and intrinsic criteria for nonlinear feedback design, \TAC, vol. 62, pp. 3046--3053, 2017.

\bibitem{MARTOM}
R. Marino and P. Tomei, {\em Nonlinear Control Design: Geometric, Adaptive and Robust}, Prentice-Hall, London, 1995. 

\bibitem{MIL}
J.D. Miller, Some global inverse function theorems, {\em J. of Math. Analysis and Appl.}, vol. 100, pp. 375--384, 1984.

\bibitem{NEM}
A. Nemirovski, Prox-method with rate of convergence of $O(1/t)$ for variational inequalties with Lipschitz continuous monotone operators and smooth convex-concave saddle point problems, {\em SIAM J. Optim.}, vol. 15, pp. 229--252, 2004.

\bibitem{NIJFOSbook}
H. Nijmeijer and T.I. Fossen, (Eds.), {\em New Directions in Nonlinear Observer Design}, Springer-Verlag, London, 1998.

\bibitem{ORT2020}
R. Ortega, A. Bobtsov, N. Nikolaev, J. Schiffer and D. Dochain, Generalized parameter estimation-based observers: Application to power systems and chemical-biological reactors, \AUT, vol. 129, 109635, 2021.

\bibitem{ORTscl}
R. Ortega, A. Bobtsov, A. Pyrkin and S. Aranovskiy, A parameter estimation approach to state observation of nonlinear systems, \SCL, vol. 85, pp. 84--94, 2015.

\bibitem{PAR}
P.A. Parrilo, Semidefinite programming relaxations for semialgebraic problems, {\em Math. Program.}, vol. 96, pp. 293--320, 2003.

\bibitem{PRA}
L. Praly, On observers with state independent error Lyapunov function, {\em Proc. IFAC Symp. Nonllinear Control Syst.} (NOLCOS), St. Petersburg, Russia, pp. 1349--1354, 2001.

\bibitem{RAPMAL}
A. Rapaport and A. Maloum, Design of exponential observers for nonlinear systems by embedding, \IJRNLC, vol. 14, pp. 273--288, 2004.

\bibitem{SANPRA}
R.G. Sanfelice and L. Praly, Convergence of nonlinear observers on $\rea^n$ with a Riemannian metric (Part I), \TAC, vol. 57, pp. 1709--1722, 2011.


\bibitem{SHO}
A. Shoshitaishvili, On control branching systems with degenerate linearization, {\em IFAC Symp. Nonlinear Control Syst.}, pp. 495--500, 1992.

\bibitem{TARRAS}
T.-J. Tarn and Y. Rasis, Observers for nonlinear stochastic systems, \TAC, vol. 21, pp. 441--448, 1976.

\bibitem{TOBetal}
M.M. Tobenkin, I.R. Manchester and A. Megretski, Convex parameterization and fidelity bounds for nonlinear indentification and reduced-order modelling, \TAC, vol. 62, no. 7, pp. 3679--3686, 2017.

\bibitem{TSI}
J. Tsinias, Further results on the observers design problem, \SCL, vol. 14, pp. 411--418, 1990.




\bibitem{WANetal}
L. Wang, F. Forni, R. Ortega, Z. Liu and H. Su, Immersion and invariance stabilization of nonlinear systems via virtual and horizontal contraction, \TAC, vol. 62, pp. 4017--4022, 2017.

\bibitem{WANSLO}
W. Wang and J.-J.E. Slotine, On partial contraction analysis for coupled nonlinear oscillators, {\em Biological Cybernetics}, vol. 92, pp. 38--53, 2005.

\bibitem{YAGYAZ}
A. Yaghmaei and M.J. Yazdanpanah, Structure preserving observer design for port-Hamiltonian systems, \TAC, vol. 64, pp. 1214--1220, 2019.

\bibitem{YIetalejc}
B. Yi, R. Ortega, H. Siguerdidjane, J.E. Machado and W. Zhang, On generation of virtual outputs via signal injection: Application to observer design for electromechanical systems, {\em European J Control}, vol. 54, pp. 129--139, 2020.

\bibitem{YIetal}
B. Yi, R. Ortega and W. Zhang, On state observers for nonlinear systems: A new design and a unifying framework, \TAC, vol. 64, pp. 1193--1200, 2019.

\bibitem{YIacc}
B. Yi, R. Wang and I.R. Manchester, Reduced-order observers for nonlinear systems via contraction analysis and convex optimisation, {\em American Control Conf.}, New Orleans, USA, May 26-28, 2021.

\end{thebibliography}
\end{document}